\documentclass{aims}
\usepackage{amsmath}
\usepackage{paralist}
\usepackage{epsfig} 
\usepackage{epstopdf}
\usepackage[colorlinks=true]{hyperref}

\usepackage{amsfonts,amssymb,stmaryrd,amssymb}
\usepackage{graphicx,color}
\usepackage[svgnames]{xcolor}
\usepackage{multimedia}
\usepackage{float}
\usepackage{placeins}

\hypersetup{urlcolor=blue, citecolor=red}

\def\currentvolume{X}
 \def\currentissue{X}
  \def\currentyear{20XX}
   \def\currentmonth{XX}
    \def\ppages{X--XX}
     \def\DOI{10.3934/xx.xx.xx.xx}

\newtheorem{theorem}{Theorem}[section]
\newtheorem{corollary}{Corollary}

\newtheorem{lemma}[theorem]{Lemma}
\newtheorem{proposition}{Proposition}

\theoremstyle{definition}
\newtheorem{definition}[theorem]{Definition}
\newtheorem{remark}{Remark}

\newtheorem{TTheorem}{Theorem}[section]

\newcommand \p {\partial}
\newcommand \dist {\mathrm{dist}}
\newcommand \com {\mathrm{cof}}

\newcommand \w {\mathrm{w}}
\newcommand \R {\mathbb{R}}

\renewcommand \L {\mathrm{L}}

\renewcommand \H {\mathrm{H}}

\newcommand \I {\mathrm{I}}
\newcommand \Id {\mathrm{Id}}

\renewcommand \d {\mathrm{d}}

\renewcommand \div {\mathrm{div}}
\renewcommand \det {\mathrm{det}}

\newcommand \bphi {\boldsymbol{\mathrm{\phi}}}

\title[Stabilization of a fluid-solid system: Part I]
      {Stabilization of a fluid-solid system, by the deformation of the self-propelled solid.\\ Part I: The linearized system.}

\author[S\'ebastien Court]{}



\email{sebastien.court@math.univ-toulouse.fr}

\thanks{This work is partially supported by the ANR-project CISIFS, 09-BLAN-0213-03.}

\begin{document}
\maketitle

\centerline{\scshape S\'ebastien Court }
\medskip
{\footnotesize
 \centerline{Institut de Math\'ematiques de Toulouse}
   \centerline{Universit\'e Paul Sabatier}
   \centerline{118 route de Narbonne}
   \centerline{31062 Toulouse Cedex 9, FRANCE}
}

\bigskip

 \centerline{(Communicated by the associate editor name)}

\begin{abstract}
This paper is the first part of a work which consists in proving the stabilization to zero of a fluid-solid system, in dimension 2 and 3. The considered system couples a deformable solid and a viscous incompressible fluid which satisfies the incompressible Navier-Stokes equations. By deforming itself, the solid can interact with the environing fluid and then move itself. The control function represents nothing else than the deformation of the solid in its own frame of reference. We there prove that the velocities of the linearized system are stabilizable to zero with an arbitrary exponential decay rate, by a boundary deformation velocity which can be chosen in the form of a feedback operator. We then show that this boundary feedback operator can be obtained from an internal deformation of the solid which satisfies the linearized physical constraints that a {\it self-propelled} solid has to satisfy.
\end{abstract}

\section{Introduction}
In this two-part work we are interested in the way a solid immersed in a viscous incompressible fluid (in dimension 2 or 3) can deform itself and then interact with the environing fluid in order to stabilize exponentially to zero the velocity of the fluid and also its own velocities. The domain occupied by the solid at time $t$ is denoted by $\mathcal{S}(t)$. We assume that $\mathcal{S}(t) \subset \mathcal{O}$, where $\mathcal{O}$ is a bounded smooth domain. The fluid surrounding the solid occupies the domain $\mathcal{O} \setminus \overline{\mathcal{S}(t)} = \mathcal{F}(t)$.
%

\vspace*{10pt}
\begin{figure}[!h]
\begin{center}
\includegraphics[scale = 0.25, trim = 3.5cm 1.8cm 1.5cm 1cm, clip]{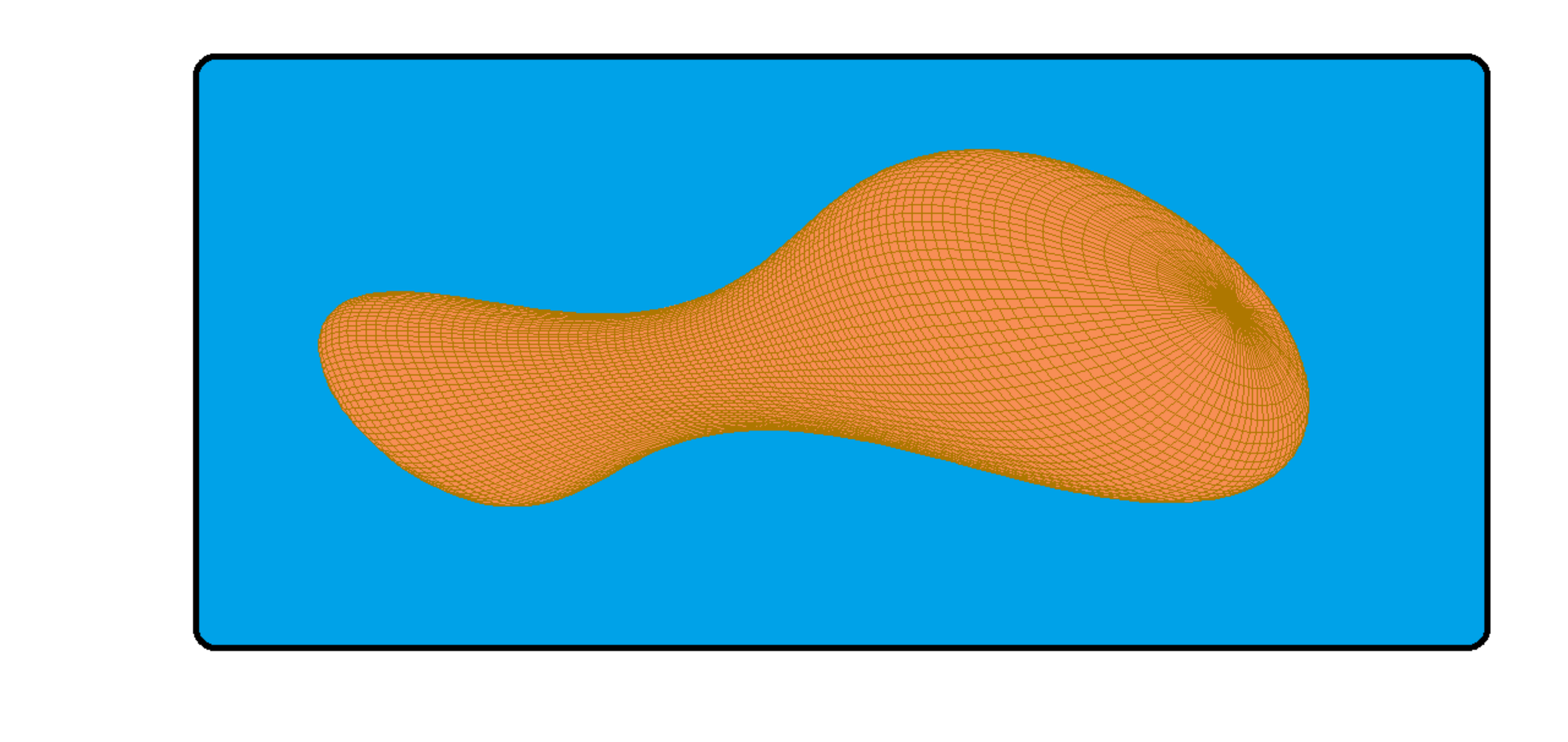}%
\end{center}
\end{figure}

\begin{picture}(0,0)(-100,-90)
\put(0,0){\makebox(0,0)[lb]{\smash{{\color[rgb]{0,0,0}\begin{huge}$\large{\mathcal{F}(t)}$\end{huge}}}}}
\put(80,-30){\makebox(0,0)[lb]{\smash{{\color[rgb]{0,0,0}\begin{huge}$\mathcal{S}(t)$\end{huge}}}}}
\put(-20,25){\makebox(0,0)[lb]{\smash{{\color[rgb]{0,0,0}\begin{large}$\mathcal{O} = \mathcal{F}(t) \cup \overline{S(t)} \subset \R^2 \text{ or } \R^3.$\end{large}}}}}
\end{picture}


\FloatBarrier

\subsection{Presentation of the model}
The movement of the solid in the inertial frame of reference is described through the time by a Lagrangian mapping denoted by $X_{\mathcal{S}}$, so we have
\begin{eqnarray*}
\mathcal{S}(t) & = & X_{\mathcal{S}}(\mathcal{S}(0),t), \quad t \geq 0.
\end{eqnarray*}
The mapping $X_{\mathcal{S}}(\cdot,t)$ can be decomposed as follows
\begin{eqnarray*}
X_{\mathcal{S}}(y,t) & = & h(t) + \mathbf{R}(t)X^{\ast}(y,t), \quad y \in \mathcal{S}(0),
\end{eqnarray*}
where the vector $h(t)$ describes the position of the center of mass and $\mathbf{R}(t)$ is the rotation associated with the angular velocity of the solid. In dimension 3 the angular velocity is a vector field whereas it is only a scalar function in dimension 2. However, $\R^2$ can be immersed in $\R^3$ and this scalar function can be read on the third component of a 3D-vector. More generally in this work, since all the calculations made in dimension 3 make sense in dimension 2, we will consider only vector fields of $\R^3$ and matrix fields of $\R^{3\times3}$. For instance, $\omega$ and $\mathbf{R}$ are related to each other through the following Cauchy problem
\begin{eqnarray*}
\left\{\begin{array} {ccccc}
\displaystyle \frac{\d \mathbf{R}}{\d t} & = & \mathbb{S}\left( \omega\right) \mathbf{R} \\
\mathbf{R}(0) & = & \I_{\R^3}
\end{array} \right.,
\quad \text{with }
\mathbb{S}(\omega) = \left(
\begin{matrix}
0 & -\omega_3 & \omega_2 \\
\omega_3 & 0 & -\omega_1 \\
-\omega_2 & \omega_1 & 0
\end{matrix} \right), \label{rotation}
\end{eqnarray*}
\footnote{The notation $\I_{\R^3}$ will represent the identity matrix of $\R^{3\times3}$.}where in dimension 2 we have $\omega = \omega_3$ and $\omega_1 = \omega_2 = 0$.\\
The couple $(h(t),\mathbf{R}(t))$ describes the position of the solid and is unknown, whereas the mapping $X^{\ast}(\cdot,t)$ can be imposed. This latter represents the deformation of the solid in its own frame of reference and is considered as the control function on which we can act physically. When this Lagrangian mapping $X^{\ast}(\cdot,t)$ is invertible, we can link to it an Eulerian velocity $w^{\ast}$ through the following Cauchy problem
\begin{eqnarray*}
\frac{\p X^{\ast}}{\p t}(y,t) = w^{\ast}(X^{\ast}(y,t),t), \quad X^{\ast}(y,0) = y-h(0), \quad y\in \mathcal{S}(0). \label{cauchystar}
\end{eqnarray*}
Without loss of generality, we assume that $h(0) = 0$, for a sake of simplicity. If $Y^{\ast}(\cdot,t)$ denotes the inverse of $X^{\ast}(\cdot,t)$, we have
\begin{eqnarray*}
 w^{\ast}(x^{\ast},t) & = & \frac{\p X^{\ast}}{\p t}(Y^{\ast}(x^{\ast},t),t), \quad x^{\ast} \in \mathcal{S}^{\ast}(t)=X^{\ast}(\mathcal{S}(0),t). \label{wsuperstar}
\end{eqnarray*}
The fluid flow is described by its velocity $u$ and its pressure $p$ which are assumed to satisfy the incompressible Navier-Stokes equations. For $X^{\ast}$ satisfying a set of hypotheses given further, the system which governs the dynamics between the fluid and the solid is the following
\begin{eqnarray}
\frac{\p u}{\p t} - \nu \Delta u + (u\cdot \nabla )u + \nabla p  =  0 , & \quad &
x \in \mathcal{F} (t),\quad t\in (0,\infty), \label{prems} \\
\div \ u   =  0 ,  & \quad &  x\in \mathcal{F} (t), \quad t\in (0,\infty), \label{deus}
\end{eqnarray}
\begin{eqnarray}
u  =  0 , & \quad & x \in \p \mathcal{O} ,\quad t\in (0,\infty),  \label{trois} \\
u  =  h'(t) +  \omega (t)\wedge(x-h(t)) + w(x,t) , & \quad &  x\in \p \mathcal{S}(t),\quad t\in
(0,\infty), \label{quatre}
\end{eqnarray}
\begin{eqnarray}
M h''(t)  =  - \int_{\p \mathcal{S}(t)} \sigma(u,p) n \d \Gamma  , & \quad & t\in (0,\infty), \label{cinq} \\
\left(I\omega\right)' (t)  = - \int_{\p \mathcal{S}(t)} (x-h(t))\wedge  \sigma(u,p) n \d \Gamma, & \quad & t\in(0,\infty),
\label{six}
\end{eqnarray}
\begin{eqnarray}
u(y,0)  =  u_0 (y), \  y\in \mathcal{F}(0) ,\quad h'(0)=h_1 \in \R^3 ,\quad \omega(0) = \omega_0 \in \R^3 ,
\end{eqnarray}
where
\begin{eqnarray}
\mathcal{S}(t) = h(t) + \mathbf{R}(t)X^{\ast}(\mathcal{S}(0),t), & \quad & \mathcal{F}(t) = \mathcal{O} \setminus \overline{\mathcal{S}(t)},
\end{eqnarray}
and where the velocity $w$ is defined by the following change of frame
\begin{eqnarray}
w(x,t) & = & \mathbf{R}(t)\ w^{\ast}\left(\mathbf{R}(t)^{T}(x-h(t)), t\right), \quad x\in \mathcal{S}(t). \label{wwwstar} \label{ders}
\end{eqnarray}
The symbol $\wedge$ denotes the cross product in $\R^3$. The linear map $\omega \wedge \cdot$ can be represented by the matrix $\mathbb{S}(\omega)$. In equations \eqref{cinq} and \eqref{six}, the mass of the solid $M$ is constant, whereas the inertia moment depends {\it a priori} on time. In dimension 2 the inertia moment is a scalar function which can be read on the inertia matrix given by
\begin{eqnarray*}
I(t) & = & \left(\int_{\mathcal{S}(t)} \rho_{\mathcal{S}}(x,t) \left|x-h(t)\right|^2 \d x\right)\I_{\R^3}.
\end{eqnarray*}
In dimension 3 it is a tensor written as
\begin{eqnarray*}
I(t) & = & \int_{\mathcal{S}(t)} \rho_{\mathcal{S}}(x,t) \left( |x-h(t)|^2 \I_{\R^3} - (x-h(t))\otimes (x-h(t)) \right) \d x.
\end{eqnarray*}
The quantity $\rho_{\mathcal{S}}$ denotes the density of the solid, and obeys the principle of mass conservation
\begin{eqnarray*}
\rho_{\mathcal{S}}(X_{\mathcal{S}}(y,t),t) & = & \frac{\rho_{\mathcal{S}}(y,0)}{\det \left(\nabla X_{\mathcal{S}}(y,t)\right)},
\quad y \in \mathcal{S}(0),
\end{eqnarray*}
where $\nabla X_{\mathcal{S}}$ denotes the Jacobian matrix of the mapping $X_{\mathcal{S}}$.
For a sake of simplicity we assume that the solid is homogeneous at time $t=0$:
\begin{eqnarray*}
\rho_{\mathcal{S}}(y,0) & = & \rho_{\mathcal{S}} > 0.
\end{eqnarray*}
In system \eqref{prems}--\eqref{ders}, $\nu$ is the kinematic viscosity of the fluid and the normalized vector $n$ is the normal at $\p \mathcal{S}(t)$ exterior to $\mathcal{F}(t)$. It is a coupled system between the incompressible Navier-Stokes equations and the differential equations \eqref{cinq}-\eqref{six} given by the Newton's laws. The coupling is in particular made in the fluid-structure interface, through the equality of velocities \eqref{quatre} and through the Cauchy stress tensor given by
\begin{eqnarray*}
\sigma(u,p)  =  2\nu D(u) - p \ \Id  =  \nu\left( \nabla u + \left(\nabla u\right)^{T} \right) - p \ \Id.
\end{eqnarray*}
Indeed, the Dirichlet condition \eqref{quatre} partially imposed by the deformation of the solid (through the velocity $w$) influences the behavior of the fluid whose the response is the quantity $\sigma(u,p)n$ in the fluid-solid interface. It represents the force that the fluid applies on the solid, and then it determines the global dynamics of the solid (through equations \eqref{cinq} et \eqref{six}) and thus its position.\\
The problem is the following: What is the deformation $X^{\ast}$ of the solid we have to impose in order to stabilize the environing fluid and thus induce a behavior of the fluid which stabilizes the velocities of the solid? We shall assume a set of hypotheses on the control function $X^{\ast}$, that we state as follows:
\begin{description}
\item[H1] For all $t\geq 0$, $X^{\ast}(\cdot,t)$ is a $C^{1}$-diffeomorphism from $\mathcal{S}(0)$ onto $ \mathcal{S}^{\ast}(t)$. \\
\item[H2] In order to respect the incompressibility condition given by \eqref{deus}, the volume of the whole solid has to be preserved through the time. That is equivalent to assume that
\begin{eqnarray}
 \int_{\p \mathcal{S}(0)} \frac{\p X^{\ast}}{\p t} \cdot \left( \com \nabla X^{\ast} \right)n\d \Gamma =  0, \label{const1}
\end{eqnarray}
\textcolor{black}{where $\com \mathbf{A}$ denotes the cofactor matrix associated with some matrix field $\mathbf{A}$. Let us remind that when this matrix is invertible the following property holds:
\begin{eqnarray*}
\com \mathbf{A}^T & = & (\det \mathbf{A})\mathbf{A}^{-1}.
\end{eqnarray*}
}
\item[H3] \textcolor{black}{The deformation of the solid does not modify its linear momentum}, which leads us to assume that
\begin{eqnarray}
\int_{\mathcal{S}(0)} \rho_{\mathcal{S}}(y,0) X^{\ast}(y,t) \d y & = & 0. \label{const2}
\end{eqnarray}
\item[H4] \textcolor{black}{The deformation of the solid does not modify its angular momentum}, which leads us to assume that
\begin{eqnarray}
\int_{\mathcal{S}(0)} \rho_{\mathcal{S}}(y,0) X^{\ast}(y,t)\wedge \frac{\p X^{\ast}}{\p t}(y,t) \d y & = & 0. \label{const3}
\end{eqnarray}
\end{description}
Imposing constraints \eqref{const2} and \eqref{const3} enables us to get the two following constraints on the undulatory velocity $w$
\begin{eqnarray}
\int_{\mathcal{S}(t)} \rho_{\mathcal{S}}(x,t)w(x,t)  \d y & = & 0, \label{const12} \\
\int_{\mathcal{S}(t)} \rho_{\mathcal{S}}(x,t)(x-h(t))\wedge w(x,t)  \d y & = & 0. \label{const13}
\end{eqnarray}
As equations \eqref{cinq} and \eqref{six} are written, the constraints \eqref{const12} and \eqref{const13} are implicitly satisfied in system \eqref{prems}--\eqref{ders}. Hypotheses {\bf H3} and {\bf H4} are made to guarantee the {\it self-propelled} nature of the motion of the solid, that means no other help than its own deformation enables it to interact and to move in the surrounding fluid.\\
The existence of global-in-time strong solutions for system \eqref{prems}--\eqref{ders} has been studied in \cite{SMSTT} in dimension 2 and more recently in \cite{Court} in dimension 3. In particular, this existence in dimension 3 is conditioned by the smallness of the data, namely the initial condition $(u_0,h_1,\omega_0)$ and the displacement of the solid $X^{\ast} - \Id_{\mathcal{S}}$\footnote{The notation $\Id_{\mathcal{S}}$ will represent the identity mapping of $\mathcal{S}$.} (in some Sobolev spaces).

\subsection{The linearized problem}
For the full nonlinear system \eqref{prems}--\eqref{ders}, the equations are written in the Eulerian configuration, and thus we are lead to think that the Eulerian velocity $w^{\ast}$ is the more suitable quantity to be chosen as a control function (instead of $X^{\ast}$). But such a mapping is defined on the domain $\mathcal{S}^{\ast}(t)$, which is itself defined by $X^{\ast}(\cdot,t)$. Moreover, the study of such a nonlinear system is based on the preliminary study of the corresponding linearized system which is
\begin{eqnarray}
\frac{\p U}{\p t}  - \div \ \sigma(U,P) = 0, & \quad &  \textrm{in $\mathcal{F}(0)\times (0,\infty)$}, \label{lin1} \\
\div \  U = 0, & \quad & \textrm{in $\mathcal{F}(0)\times (0,\infty)$},
\end{eqnarray}
\begin{eqnarray}
U  =  0 , & \quad & \textrm{in $\p \mathcal{O} \times (0,\infty)$}, \\
U  =  H'(t) + \Omega(t) \wedge y + \zeta(y,t) , & \quad & y \in \p \mathcal{S}(0) ,\quad t\in (0,\infty),
\end{eqnarray}
\begin{eqnarray}
M H''(t)  = - \int_{\p \mathcal{S}} \sigma(U,P) n \d \Gamma , & \quad & t\in (0,\infty), \label{lin5} \\
I_0\Omega'(t)  = -  \int_{\p \mathcal{S}} y \wedge \sigma(U,P) n \d \Gamma , & \quad &  t\in(0,\infty), \label{lin6}
\end{eqnarray}
\begin{eqnarray}
U(y,0)  =  u_0 (y), \  y \in \mathcal{F}(0), \quad H'(0)=h_1 \in \R^3 ,\quad \Omega(0) = \omega_0 \in \R^3, \label{lin10}
\end{eqnarray}
and where the more suitable control to be chosen is the function $\zeta$, related to the Lagrangian velocity $\displaystyle \frac{\p X^{\ast}}{\p t}$ by
\begin{eqnarray*}
\zeta & = & e^{\lambda t}\frac{\p X^{\ast}}{\p t}_{\left| \p \mathcal{S} \right.}.
\end{eqnarray*}
\textcolor{black}{In this system $I_0$ denotes the inertia matrix of the solid at time $t=0$.\\
}
Notice that the constraints \eqref{const1} and \eqref{const3} are nonlinear with respect to the mapping $X^{\ast}$. We linearize them when we consider the linear system \eqref{lin1}--\eqref{lin10}. For this linear system, the constraint induced by Hypothesis {\bf H1} can be relaxed, since we only consider mappings $X^{\ast}(\cdot ,t)$ continuous in time and such that $X^{\ast}(\cdot,0) = \Id_{\mathcal{S}}$. Thus the notion of {\it admissible control} for this linearized problem is made precise in Definition \ref{deflincontrol} \textcolor{black}{(see below)}.

\subsection{The main result and the strategy}
The main result of this first part is Theorem \ref{thstablinX}, which is equivalent to the following one:
\begin{TTheorem}
Assume that $\dist( \p \mathcal{O} ; \mathcal{S}(0)) > 0$. For all $(u_0,h_1,\omega_0)$ satisfying $u_0 \in \mathbf{H}^1(\mathcal{F}(0))$ and the following compatibility conditions
\begin{eqnarray*}
\left\{ \begin{array} {ll}
\div \ u_0  =  0 & \text{ in } \mathcal{F}(0), \\
u_0 = 0 & \text{ on } \p \mathcal{O}, \\
u_0 = h_1 + \omega_0 \wedge y & \text{ on } \p \mathcal{S}(0),
\end{array} \right.
\end{eqnarray*}
system \eqref{lin1}--\eqref{lin10} is stabilizable with an arbitrary exponential decay rate $\lambda > 0$, that is to say that for all $\lambda >0$ there exists a boundary control $\zeta \in \L^2(0,\infty;\mathbf{H}^{3/2}(\p \mathcal{S}))$ and a positive constant C\footnote{In the following the symbols $C$, $\tilde{C}$ or $\overline{C}$ will denote some generic positive constants which do not depend on time or on the unknowns.} depending only on $(u_0,h_1,\omega_0)$ such that for all $t \geq 0$ the solution $(U,H',\Omega)$ of system \eqref{lin1}--\eqref{lin10} satisfies
\begin{eqnarray*}
\| (U(\cdot,t),H'(t),\Omega(t) \|_{\mathbf{L}^2(\mathcal{F}(0))\times \R^3 \times \R^3} & \leq & C\exp(-\lambda t).
\end{eqnarray*}
\end{TTheorem}
For proving this theorem we study the system that has to satisfy the functions
\begin{eqnarray*}
\hat{U}(\cdot,t) = e^{\lambda t}U(\cdot,t), \quad \hat{P}(\cdot,t) = e^{\lambda t}P(\cdot,t),
\quad \hat{H}'(t) = e^{\lambda t}H'(t), \quad \hat{\Omega}(t) = e^{\lambda t}\Omega(t),
\end{eqnarray*}
and the goal is then to prove that there exists a control $\zeta$ such that this system admits a solution $(\hat{U},\hat{H}',\hat{\Omega})$ bounded in some infinite time horizon space. The strategy we follow is globally the same as the one used in \cite{JPR3}, at least for the linearized problem. It first consists in rewriting the full nonlinear system in space domains which do not depend on time anymore, by using a change of variables and a change of unknowns. Then we can make appear all the nonlinearities (specially those which are due to the variations of the geometry through the time) and we can set properly the linearized system. The second step of the proof consists in formulating the linearized system in terms of operators where the pressure is actually eliminated and encodes a {\it mass-added effect}. This writing enables us to define an analytic semigroup of contraction generated by an operator which presents interesting spectral properties: Indeed the unstable modes that we have to stabilize are actually countable and in finite number. Besides, we prove by a unique continuation argument the approximate controllability of this linearized system. Thus, in order to define the boundary control that stabilizes the full linearized system, it is sufficient to consider a finite-dimension linear system for which the approximate controllability is equivalent to the feedback stabilizability. The aforementioned control can be defined on a finite-dimension space in a feedback operator form, what will be useful for proving the stabilization of the full nonlinear system.\\
With regards to the methods, a novelty is the means provided in a last section which enables us to define from a boundary control an internal deformation satisfying the linearized constraints. This result too will be useful for the definition of a deformation of the solid - which has to satisfy the nonlinear constraints - that stabilizes the nonlinear system: Considering a deformation which satisfies in a first time the linearized constraints is not necessarily from a mathematical point of view, but the method with which we obtain it will be important for defining a deformation satisfying the nonlinear constraints (see section 5 of Part II); Besides, from a physical point of view, considering for the linearized system a deformation which satisfies the linear constraints is relevant since it ensures the conservation of the momenta for the whole fluid-solid linear system.\\
The idea of considering first the linearized problem relies on the fact that for small perturbations (that is to say for small initial conditions $u_0$, $h_1$ and $\omega_0$) the behavior of the nonlinear system is close to the one of the linearized system.\\
\hfill \\
Thus the same statement will be proven in the second part of this work for the unknowns of the full nonlinear system \eqref{prems}--\eqref{ders}. The result is nonintuitive: It says somewhat that all the fluid in which the solid swims can be stabilized just by the help of this swimmer, at an intermediate Reynolds number. This kind of problem has been investigated in \cite{Khapalov1, Khapalov2} for instance, where it is considered other types of fluid-swimmer systems. The same kind of purpose has been also investigated at a low Reynolds number in \cite{SanMartin2} and more recently in \cite{Loheac}, for the Stokes system, and in \cite{Munnier1, Munnier2} in the case of a perfect fluid. The control of the motion of a boat at a high Reynolds number has been recently studied in \cite{Glass}. Besides, the same kind of techniques that we use in this work have been used for other coupled systems involving the incompressible Navier-Stokes equations; Let us cite \cite{Avalos} and \cite{JPR3} for instance, where the stabilization of other fluid-structure problems is proven.

\subsection{Plan}
Definitions and notation are given in section \ref{secdef}. The linearized system is studied in section \ref{linearsec} where it is rewritten through an operator formulation for which we prove useful properties. The approximate controllability of this linearized system is proven in section \ref{secapproxcont}. It leads to the main result of this paper, namely the feedback stabilization of the linearized system in section \ref{secfeedback}. Finally in section \ref{secAhAh} we provide a means to recover an internal deformation of the solid from a feedback boundary control which lives only in the fluid-solid interface. This final section is the transition to the second part of this work where the considered control is an internal deformation of the solid.

\section{Definitions and notation} \label{secdef}
We denote by $\mathcal{F} = \mathcal{F}(0)$ the domain occupied by the fluid at time $t=0$, and by $\mathcal{S}= \mathcal{S}(0)$ the domain occupied by the solid at $t=0$. We assume that $\mathcal{S}$ is smooth enough. We set
\begin{eqnarray*}
S^0_{\infty} = \mathcal{S} \times (0,\infty), & \quad & Q_{\infty}^0 = \mathcal{F} \times (0,\infty).
\end{eqnarray*}
Let us introduce some functional spaces. \textcolor{black}{Concerning the classical Sobolev spaces}, we use the notation $\mathbf{H}^s(\Omega) = [\H^s(\Omega)]^d$ or $[\H^s(\Omega)]^{d^k}$, for some positive integer $k$, for all bounded domain $\Omega$ of $\R^2$ or $\R^3$. We classically define
\begin{eqnarray*}
\mathbf{V}^0_n(\mathcal{F}) & = & \left\{ \phi \in  \mathbf{L}^2(\mathcal{F}) \mid \div \
\phi = 0 \text{ in } \mathcal{F}, \ \phi \cdot n =0 \text{ on $\p \mathcal{O}$}\right\}, \\
\mathbf{V}^1_n(\mathcal{F}) & = & \left\{ \phi \in  \mathbf{H}^1(\mathcal{F}) \mid \div \
\phi = 0 \text{ in } \mathcal{F}, \ \phi \cdot n =0 \text{ on $\p \mathcal{O}$}\right\}, \\
\H^{2,1}(Q_{\infty}^0) & = & \L^2(0,\infty;\mathbf{H}^{2}(\mathcal{F})) \cap \H^{1}(0,\infty;\mathbf{L}^2(\mathcal{F})),
\end{eqnarray*}
Let us keep in mind the continuous embedding:
\begin{eqnarray*}
\H^{2,1}(Q_{\infty}^0) & \hookrightarrow & \L^{\infty}(0,\infty; \mathbf{H}^1(\mathcal{F})).
\end{eqnarray*}

\noindent We finally set the spaces dealing with compatibility conditions
\begin{eqnarray*}
\mathbf{H}^0_{cc} & = & \left\{(u_0,h_1,\omega_0) \in \mathbf{V}^0_n(\mathcal{F}) \times \R^3 \times \R^3 \mid \
u_0 = h_1 +\omega_0 \wedge y \text{ on } \p \mathcal{S} \right\}, \\
\mathbf{H}^1_{cc} & = & \left\{(u_0,h_1,\omega_0) \in \mathbf{V}^1_n(\mathcal{F}) \times \R^3 \times \R^3 \mid \
u_0 = h_1 +\omega_0 \wedge y \text{ on } \p \mathcal{S} \right\}.
\end{eqnarray*}
For more simplicity, we assume that the density $\rho_{\mathcal{S}}$ at time $t=0$ is constant with respect to the space:
\begin{eqnarray*}
\rho_{\mathcal{S}}(y,0) = \rho_{\mathcal{S}} > 0.
\end{eqnarray*}
We assume without loss of generality that $h_0 = 0$. This implies in particular
\begin{eqnarray*}
\int_{\mathcal{S}}y \d y & = & 0.
\end{eqnarray*}
Let us state the conditions we shall assume on mappings chosen as control functions.
\begin{definition} \label{deflincontrol}
\textcolor{black}{Let $\lambda >0$ be a desired exponential decay rate.} A deformation $X^{\ast}$ is said admissible for the linear system \eqref{lin1}--\eqref{lin10} if
\begin{eqnarray*}
e^{\lambda t}\frac{\p X^{\ast}}{\p t} & \in & \L^2(0,\infty;\mathbf{H}^3(\mathcal{S}) \cap \H^1(0,\infty;\mathbf{H}^1(\mathcal{S}))
\end{eqnarray*}
and if for all $t \geq 0$ it satisfies the following hypotheses
\begin{eqnarray}
\int_{\p \mathcal{S}} \displaystyle \frac{\p X^{\ast}}{\p t}(y,t) \cdot n \d \Gamma(y) & = & 0, \label{const105} \\
\int_{\mathcal{S}}  X^{\ast}(y,t) \d y & = & 0, \label{const101} \\
\int_{\mathcal{S}}  y\wedge \frac{\p X^{\ast}}{\p t}(y,t) \d y & = & 0. \label{const102}
\end{eqnarray}
\end{definition}

\textcolor{black}{
\begin{remark}
In the defintion above, the hypotheses \eqref{const105}, \eqref{const101}, \eqref{const102} are the linearized versions of the constraints \eqref{const1}, \eqref{const2}, \eqref{const3} respectively, with respect to the mapping $X^{\ast}$. The expression of \eqref{const2} is already linear, whereas the two other ones can be written as follows
\begin{eqnarray*}
\int_{\p \mathcal{S}} \displaystyle \frac{\p X^{\ast}}{\p t} \cdot \left(\com \nabla X^{\ast} \right)n \d \Gamma & = &
\int_{\p \mathcal{S}} \displaystyle \frac{\p X^{\ast}}{\p t} \cdot n \d \Gamma  \\
& & + \int_{\p \mathcal{S}} \displaystyle \frac{\p (X^{\ast}- y)}{\p t} \cdot \left(\com \nabla X^{\ast} - \I_{\R^3} \right)n \d \Gamma, \\
\int_{\mathcal{S}} X^{\ast}\wedge \frac{\p X^{\ast}}{\p t} \d y & = &
\int_{\mathcal{S}}  y \wedge \frac{\p X^{\ast}}{\p t} \d y 
+ \int_{\mathcal{S}}  (X^{\ast}- y)\wedge \frac{\p (X^{\ast}- y)}{\p t} \d y.
\end{eqnarray*}
Lipschitz properties for the quadratic residues - like the one involving $(\com \nabla X^{\ast} - \I_{\R^3})$ for instance - will be studied in Part II. The transition between deformations which satisfy the linearized constraints and deformations which satisfy the nonlinear constraints is mainly explained in section 5.3 of Part II.
\end{remark}
}

\section{Operator formulation: Definition of an analytic semigroup} \label{linearsec}
\textcolor{black}{System \eqref{prems}--\eqref{ders} is strongly coupled, in particular from a geometrical point of view. A first work would consists in rewriting it in space domains which do not depend on time anymore, by using a change of unknowns defined with the help of a change of variables. For more details we refer to section 3.2 of Part II, where it is explained why in this section we study the following linear system, which is nothing else than system \eqref{lin1}--\eqref{lin10} with a resolvent term:}
\begin{eqnarray}
\frac{\p u}{\p t} - \lambda u - \div \ \sigma(u,p) & = & 0, \quad \textrm{in $\mathcal{F}\times (0,\infty)$}, \label{premss}\\
\div \  u & = & 0, \quad \textrm{in $\mathcal{F}\times (0,\infty)$},
\end{eqnarray}
\begin{eqnarray}
u  =  0 , &\quad & \textrm{in $\p \mathcal{O}\times (0,\infty)$}, \\
u  =  h'(t) + \omega(t) \wedge y + \zeta(y,t) , & \quad & y \in \p \mathcal{S} ,\quad t\in (0,\infty),
\end{eqnarray}
\begin{eqnarray}
M h''-\lambda M h' = - \int_{\p \mathcal{S}} \sigma(u,p) n \d \Gamma , & \quad & \textrm{in } (0,\infty),  \label{cinqss}\\
I_0\omega' - \lambda I_0\omega = -  \int_{\p \mathcal{S}} y \wedge \sigma(u,p) n \d \Gamma , & \quad & \textrm{in } (0,\infty), \label{sixss}
\end{eqnarray}
\begin{eqnarray}
u(y,0) & = & u_0 (y), \  y \in \mathcal{F}, \quad h'(0)=h_1 \in \R^3 ,\quad \omega(0) = \omega_0 \in \R^3. \label{derss}
\end{eqnarray}

\textcolor{black}{Note that in this linear system the control (initially chosen as the deformation of the solid) appears only on the boundary $\p \mathcal{S}$ through the function $\zeta$ which stands for $\displaystyle e^{\lambda t}\frac{\p X^{\ast}}{\p t}$. Thus the problem of stabilization of this linear system is reduced to a problem of boundary stabilization.}

\subsection{Introduction of some operators} \label{secdefop}
For what follows, we need to introduce some operators. Let us first remind the notation of the stress tensor of some vector field $u$
\begin{eqnarray*}
D(u) & = & \frac{1}{2}\left(\nabla u + \nabla u ^T\right).
\end{eqnarray*}
and let us denote the Hessian matrix operator as
\begin{eqnarray*}
H & = & \nabla^2.
\end{eqnarray*}
For $h'\in \R^3$ and $\omega \in \R^3$, we define $N(h')$ and $\hat{N}(\omega)$ as being the respective solutions $q$ and $\hat{q}$ of the Neumann problems
\begin{eqnarray*}
\Delta q = 0 \text{ in } \mathcal{F}, \quad \frac{\p q}{\p n} = h'\cdot n \text{ on } \p \mathcal{S}, \quad \frac{\p q}{\p n} = 0 \text{ on } \p \mathcal{O}, \\
\Delta \hat{q} = 0 \text{ in } \mathcal{F}, \quad \frac{\p \hat{q}}{\p n} = (\omega \wedge y)\cdot n \text{ on } \p \mathcal{S}, \quad \frac{\p \hat{q}}{\p n} = 0 \text{ on } \p \mathcal{O}.
\end{eqnarray*}
For $\varphi \in \mathbf{H}^{1/2}(\p \mathcal{S})$, we define $L_0\varphi = \w$ as being the solution of the Stokes problem
\begin{eqnarray*}
-\nu \Delta \w + \nabla \psi = 0 \text{ in } \mathcal{F}, & \quad & \div \ \w = 0  \text{ in } \mathcal{F}, \\
\w = 0  \text{ on } \p \mathcal{O}, & \quad & \w = \varphi  \text{ on } \p \mathcal{S}.
\end{eqnarray*}
Similarly we define $\hat{L}_0\varphi = \hat{\w}$ as being the solution of
\begin{eqnarray*}
-\nu \Delta \hat{\w} + \nabla \hat{\psi} = 0 \text{ in } \mathcal{F}, &\quad & \div \ \hat{\w} = 0  \text{ in } \mathcal{F}, \\
\hat{\w} = 0  \text{ on } \p \mathcal{O}, & \quad & \hat{\w} = \varphi \wedge y  \text{ on } \p \mathcal{S}.
\end{eqnarray*}
The operators $L_0$ and $\hat{L}_0$ are called {\it lifting operators}. We also define the following integration operators
\begin{eqnarray*}
\mathcal{C}\varphi  =  \int_{\p \mathcal{S}} \varphi n\d \Gamma , & \quad &
\hat{\mathcal{C}}\varphi  =  \int_{\p \mathcal{S}} y\wedge \varphi n\d \Gamma .
\end{eqnarray*}
We denote by $\mathbb{P} : \mathbf{L}^2(\mathcal{F}) \rightarrow \mathbf{V}^0_n(\mathcal{F})$ the so-called Leray or Helmholtz operator, which is the orthogonal projection induced by the decomposition
\begin{eqnarray*}
\mathbf{L}^2(\mathcal{F}) = \mathbf{V}^0_n(\mathcal{F}) \oplus \nabla \mathbf{H}^1(\mathcal{F}).
\end{eqnarray*}
Then we denote in $\mathbf{V}^0_n$ the classical Stokes operator by
\begin{eqnarray*}
A_0 & = & \nu \mathbb{P} \Delta,
\end{eqnarray*}
with domain $D(A_0) = \mathbf{H}^2(\mathcal{F}) \cap \mathbf{H}^1_0(\mathcal{F}) \cap \mathbf{V}^0_n(\mathcal{F})$.

\subsection{Operator formulation}
Let us first consider system \eqref{premss}--\eqref{derss} only when $\displaystyle \zeta = 0$; In that case we denote by $(v,p,h',\omega)$ the unknowns. But this system can be transformed into a system whose unknowns are only $(v,h',\omega)$. Indeed, by following the method which is used in \cite{JPR1} or \cite{JPR3} for instance, the pressure $p$ can be eliminated in the equations \eqref{premss}, \eqref{cinqss} and \eqref{sixss}. By this means we obtain that $p$ can be written
\begin{eqnarray*}
p & = & \pi - \frac{\p q }{\p t},
\end{eqnarray*}
where $\pi$ is solution of the following Neumann problem
\begin{eqnarray*}
\Delta \pi(t) = 0 \text{ in } \mathcal{F}, & \quad & \frac{\p \pi(t)}{\p n} = \nu \mathbb{P}\Delta v(t)\cdot n \text{ on } \p \mathcal{F},
\end{eqnarray*}
and $q$ is solution of this other Neumann problem which involves the boundary conditions
\begin{eqnarray*}
\Delta q(t) = 0 \text{ in } \mathcal{F}, \quad \frac{\p q(t)}{\p n} = (h'(t)+\omega(t)\wedge y)\cdot n \text{ on } \p \mathcal{S}, \quad \frac{\p q(t)}{\p n} = 0 \text{ on } \p \mathcal{O}.
\end{eqnarray*}
Moreover, $\nabla p$ can be expressed through a lifting; More precisely, we have
\begin{eqnarray*}
\nabla p & = & (-A_0)\mathbb{P}L_0(h'+\omega\wedge y) \quad \text{ in } \mathcal{F}.
\end{eqnarray*}
Thus, we can split system \eqref{premss}--\eqref{derss} into two systems, one satisfied by $(\mathbb{P}v,h',\omega)$, and the other one satisfied by $(\Id - \mathbb{P})v$. Explicitly, by denoting $V=(\mathbb{P}v,h',\omega)^T$, system \eqref{premss}--\eqref{derss} can be rewritten (for $\zeta = 0$) as follows
\begin{eqnarray}
\left( \mathbb{M}_0 + \mathbb{M}_{add} \right) V' & = & \mathbb{A} V +\lambda \mathbb{M}_0 V, \label{expform0} \\
\left( \Id -\mathbb{P} \right) v & = & \left( \Id -\mathbb{P} \right)\left(L_0(h') + \hat{L}_0(\omega)\right), \label{expform}
\end{eqnarray}
with
\begin{eqnarray*}
\mathbb{M}_0 = \left[ \begin{matrix}
\Id & 0 & 0 \\
0 & M\Id & 0 \\
0 & 0 & I_0
\end{matrix}
\right], & \quad &
\mathbb{M}_{add} = \left[ \begin{matrix}
0 & 0 & 0 \\
0 & \mathcal{C}N & \mathcal{C}\hat{N} \\
0 & \hat{\mathcal{C}}N & \hat{\mathcal{C}}\hat{N}
\end{matrix}
\right],
\end{eqnarray*}
and
\begin{eqnarray*}
\mathbb{A} & = &
\left[ \begin{matrix}
A_0 & (-A_0)\mathbb{P}L_0 & (-A_0)\mathbb{P}\hat{L}_0 \\
\mathcal{C}(-2\nu D + N A_0) & -\mathcal{C}H N & -\mathcal{C}H \hat{N} \\
\hat{\mathcal{C}}(-2\nu D + N A_0) & -\hat{\mathcal{C}}H N & -\hat{\mathcal{C}}H \hat{N}
\end{matrix}
\right].
\end{eqnarray*}
The operator $\mathbb{M}_{add}$ represents a {\it mass-added} effect. We are going to see that it contributes to making the "effective mass operator" - namely $\mathbb{M}_0 + \mathbb{M}_{add}$ - self-adjoint and positive.

\subsection{Main properties of the operator $\mathbb{A}$}
Let us set
\begin{eqnarray*}
\mathbb{M} & = & \mathbb{M}_0 + \mathbb{M}_{add}.
\end{eqnarray*}

\begin{lemma}
$\mathbb{M}$ is self-adjoint and positive.
\end{lemma}

\begin{proof}
Observe that $\mathbb{M}_0$ is self-adjoint and positive. Then it is sufficient to show that $\mathbb{M}_{add}$ is self-adjoint and non-negative. Let us begin with noticing that $\mathcal{C}N$ is self-adjoint. Indeed, if $q_1$ and $q_2$ denote respectively $N(h'_1)$ and $N(h'_2)$, by using twice the Green formula we get
\begin{eqnarray*}
\begin{array} {ccccc}
\int_{\mathcal{F}} \nabla q_1 \cdot \nabla q_2 \d y & = & \displaystyle \int_{\p \mathcal{S}}(h'_1\cdot n) q_2 \d \Gamma & = & \displaystyle \int_{\p \mathcal{S}} (h'_2\cdot n) q_1 \d \Gamma \\
& = & h'_1 \cdot \mathcal{C}N(h'_2) & = & h'_2 \cdot \mathcal{C}N(h'_1).
\end{array}
\end{eqnarray*}
We can also see that $\hat{\mathcal{C}}\hat{N}$ is self-adjoint. If $\hat{q}_1$ and $\hat{q}_2$ denote respectively $\hat{N}(\omega_1)$ and $\hat{N}(\omega_2)$, by using twice the Green formula we get
\begin{eqnarray*}
\begin{array} {ccccc}
\int_{\mathcal{F}} \nabla \hat{q}_1 \cdot \nabla \hat{q}_2 \d y & = & \displaystyle \int_{\p \mathcal{S}} \omega_1 \cdot (y\wedge n) \hat{q}_2 \d \Gamma & = & \displaystyle \int_{\p \mathcal{S}} \omega_2 \cdot (y\wedge n) \hat{q}_1 \d \Gamma \\
& = & \omega_1 \cdot \hat{\mathcal{C}}\hat{N}(\omega_2) & = & \omega_2 \cdot \hat{\mathcal{C}}\hat{N}(\omega_1).
\end{array}
\end{eqnarray*}
Likewise, let us show that $(\hat{\mathcal{C}}N )^T = \mathcal{C}\hat{N}$. First, we denote $\hat{q} = \hat{N}(\omega)$, and by using twice the Green formula, we have
\begin{eqnarray*}
\begin{array} {ccccc}
\mathcal{C}\hat{N}(\omega) & = & \displaystyle \int_{\mathcal{F}}\nabla \hat{q} \d y & = & \displaystyle \int_{\p \mathcal{S}} \frac{\p \hat{q}}{\p n} y \d \Gamma \\
& = & \displaystyle \int_{\p \mathcal{S}} y \otimes y (n\wedge \omega) \d \Gamma & = & \displaystyle \left(\int_{\p \mathcal{S}} y\otimes (y\wedge n)\d \Gamma \right) \omega.
\end{array}
\end{eqnarray*}
On the other hand, if we denote $q = \hat{\mathcal{C}}N(h')$, let us notice that $q= h'\cdot y$. And then
\begin{eqnarray*}
\begin{array} {ccccc}
\hat{\mathcal{C}}N(h') & = & \displaystyle \int_{\p \mathcal{S}}(y\cdot h')(y\wedge n)\d \Gamma & = & \displaystyle \left( \int_{\p \mathcal{S}} (y\wedge n) \otimes y \d \Gamma \right) h'.
\end{array}
\end{eqnarray*}
Then we can conclude that $\mathbb{M}_{add}$ is self-adjoint. In order to prove that $\mathbb{M}_{add}$ is non-negative, let us see that the corresponding quadratic term can be written
\begin{eqnarray*}
V^T \mathbb{M}_{add} V = \int_{\mathcal{F}} \left| \nabla q + \nabla \hat{q} \right|^2 \d y \geq 0
\end{eqnarray*}
for $V = (\mathbb{P}v,h',\omega)^T$, $q = N(h')$ and $\hat{q} = \hat{N}(\omega)$.
\end{proof}


In the following, we will denote
\begin{eqnarray*}
\mathcal{A} & = & \left(\mathbb{M}_0+\mathbb{M}_{add}\right)^{-1}\mathbb{A} .
\end{eqnarray*}

\begin{proposition} \label{propsa}
The domain of $\mathcal{A}$ is
\begin{eqnarray*}
D(\mathcal{A}) & = & \mathbf{H}^1_{cc} \cap \left(\mathbf{H}^2(\mathcal{F})\times \R^3 \times \R^3\right),
\end{eqnarray*}
and $\mathbb{A} = \mathbb{M}\mathcal{A}$ is self-adjoint.
\end{proposition}

\begin{proof}
Let $V_1 = (v_1,h'_1,\omega_1)^T$ and $V_2 = (v_2,h'_2,\omega_2)^T$ lie in $D(\mathcal{A})$. We set $(F_1,F_{M_1},F_{I_1})^T = (\lambda \mathbb{M} - \mathbb{A})V_1$ and $(F_2,F_{M_2},F_{I_2})^T = (\lambda \mathbb{M} - \mathbb{A})V_2$, that is to say that for $i\in \{1,2\}$ we have
\begin{eqnarray*}
\lambda v_i - \div \ \sigma(v_i,p_i) =
F_i, & \quad & \textrm{in $\mathcal{F} $}, \\
\div \  v_i = 0, & \quad & \textrm{in $\mathcal{F}$},
\end{eqnarray*}
\begin{eqnarray*}
v_i  =  0 , &\quad & \textrm{in $\p \mathcal{O}$}, \\
v_i  =  h_i'(t) + \omega_i(t) \wedge y , & \quad & y \in \ \p \mathcal{S},
\end{eqnarray*}
\begin{eqnarray*}
\lambda M h_i'(t) = - \int_{\p \mathcal{S}} \sigma(v_i,p_i) n \d \Gamma +F_{M_i}, \\
\lambda I_0\omega_i (t) = -  \int_{\p \mathcal{S}} y \wedge \sigma(v_i,p_i) n \d \Gamma + F_{I_i},
\end{eqnarray*}
with
\begin{eqnarray*}
p_i & = & N(A_0 v_i) - \lambda N(h'_i+\omega_i \wedge y).
\end{eqnarray*}
We calculate
\begin{eqnarray*}
\langle V_2 ; (F_1,F_{M_1},F_{I_1})^T \rangle_{\mathbf{L}^2(\mathcal{F})\times \R^3 \times \R^3} = \lambda \langle V_1,V_2 \rangle_{\mathbf{L}^2(\mathcal{F})} - \int_{\mathcal{F}}v_1\cdot \div \ \sigma(v_1,p_1) & & \\
  + h'_2\cdot \int_{\p \mathcal{S}}\sigma(v_1,p_1)n\d \Gamma +\omega_2 \cdot \int_{\p \mathcal{S}}y\wedge \sigma(v_1,p_1)n\d \Gamma, & &
\end{eqnarray*}
and by integration by parts we get
\begin{eqnarray*}
& &- \int_{\mathcal{F}}v_1\cdot \div \ \sigma(v_1,p_1) + h'_2\cdot \int_{\p \mathcal{S}}\sigma(v_1,p_1)n\d \Gamma +\omega_2 \cdot \int_{\p \mathcal{S}}y\wedge \sigma(v_1,p_1)n\d \Gamma \\
 & = & - \int_{\p \mathcal{S}}v_2\cdot \sigma(v_1,p_1)n\d \Gamma + 2\nu\int_{\mathcal{S}} D(v_1):D(v_2) \\
 & & + h'_2\cdot \int_{\p \mathcal{S}}\sigma(v_1,p_1)n\d \Gamma +\omega_2 \cdot \int_{\p \mathcal{S}}y\wedge \sigma(v_1,p_1)n\d \Gamma  \\
 & = & 2\nu\int_{\mathcal{S}} D(v_1):D(v_2).
\end{eqnarray*}
Then, by swapping the roles of $V_1$ and $V_2$, it is easy to see that
\begin{eqnarray*}
\langle V_2 ; (F_1,F_{M_1},F_{I_1})^T \rangle_{\mathbf{L}^2(\mathcal{F})\times \R^3 \times \R^3} & = & \langle (F_2,F_{M_2},F_{I_2})^T ; V_1 \rangle_{\mathbf{L}^2(\mathcal{F})\times \R^3 \times \R^3} .
\end{eqnarray*}
It shows that $\lambda \mathbb{M} - \mathbb{A}$ is self-adjoint. Since $\mathbb{M}$ is self-adjoint, the proof is complete.\\
\end{proof}

\begin{proposition}
The resolvent of $\mathcal{A}$ is compact.
\end{proposition}

\begin{proof}
For $\mathbb{F} = (F,F_M,F_I)^T \in \L^2(0,\infty;\mathbf{L}^2(\mathcal{F})) \times \L^2(0,\infty;\R^3) \times \L^2(0,\infty;\R^3)$, we consider the system
\begin{eqnarray*}
\mathbb{M}\left(\lambda \Id - \mathcal{A}\right)V & = & \mathbb{F},
\end{eqnarray*}
where $V=(v,h',\omega)^T$ is the unknown. This system can be rewritten as
\begin{eqnarray*}
\lambda v - \div \ \sigma(v,p) =
F, & \quad & \textrm{in $\mathcal{F} $}, \\
\div \  v = 0, & \quad & \textrm{in $\mathcal{F}$},
\end{eqnarray*}
\begin{eqnarray*}
v  =  0 , &\quad & \textrm{in $\p \mathcal{O}$}, \\
v  =  h'(t) + \omega(t) \wedge y , & \quad & y \in \p \mathcal{S},
\end{eqnarray*}
\begin{eqnarray*}
\lambda M h'(t) = - \int_{\p \mathcal{S}} \sigma(v,p) n \d \Gamma +F_{M}, \\
\lambda I_0\omega (t) = -  \int_{\p \mathcal{S}} y \wedge \sigma(v,p) n \d \Gamma + F_{I},
\end{eqnarray*}
with
\begin{eqnarray*}
p & = & N(A_0 v) - \lambda N(h'+\omega \wedge y).
\end{eqnarray*}
By the same kind of calculations as the ones made in the proof of Proposition \ref{propsa}, this problem is equivalent to the following variational problem:
\begin{eqnarray}
\text{Find $V\in \mathbf{V}^1_n\times\R^3\times\R^3$ such that } \quad a(V,W)=l(W) \label{pbvar}
\end{eqnarray}
with $V=(v,h',\omega)^T$, $W=(w,k',\alpha)^T$, and
\begin{eqnarray*}
a(V,W) & = & \lambda \left(\int_{\mathcal{F}}v\cdot w +Mh'\cdot k' + I_0\omega\cdot \alpha \right) + 2\nu\int_{\mathcal{F}}D(v):D(w), \\
l(W) & = & \int_{\mathcal{F}}F\cdot w +F_M\cdot k' +F_I \cdot \alpha.
\end{eqnarray*}
By choosing $\lambda = 1$, we use the Lax-Milgram theorem and prove that problem \eqref{pbvar} has a unique solution. Thus $\mathbb{M}\left(\lambda\Id - \mathcal{A}\right)$ is invertible, and since $\mathbb{M}$ is positive, $\lambda\Id - \mathcal{A}$ is also invertible.
\end{proof}

The results of the two last propositions yield the following theorem.
\begin{theorem} \label{thanalsg}
The operator $(\mathcal{A}, D(\mathcal{A}))$ is the infinitesimal generator of an analytic semigroup on $\mathbf{V}^0_n(\mathcal{F}) \times \R^3 \times \R^3$, and the resolvent of $\mathcal{A}$ is compact.
\end{theorem}

\subsection{Abstract formulation of the control problem} \label{seccontrol}

\begin{proposition} \label{propoperator}
The triplet $(u,h',\omega)$ is solution of system \eqref{premss}--\eqref{derss} if and only if $U = (\mathbb{P}u,h',\omega)^T$ and $(\Id - \mathbb{P})u$ satisfy the operator formulation
\begin{eqnarray}
U' & = & \mathcal{A}_{\lambda} U + \mathcal{B}_{\lambda}\zeta , \label{control1} \\
\left( \Id -\mathbb{P} \right) u & = & \left( \Id -\mathbb{P} \right)\left(L_0(h') + \hat{L}_0(\omega)\right). \label{control2}
\end{eqnarray}
with $\mathcal{A}_{\lambda}  =  \mathcal{A} + \lambda \mathbb{M}^{-1}\mathbb{M}_0$, $\mathcal{B}_{\lambda} = \mathbb{M}^{-1}\mathbb{B}_{\lambda} = \mathbb{B}_{\lambda}$, and
\begin{eqnarray*}
\mathbb{B}_{\lambda} & = &
\left[ \begin{array} {c}
\left(\lambda \Id - A_0 \right)L_0 \\
0 \\
0
\end{array} \right].
\end{eqnarray*}
\end{proposition}

\begin{proof}
Let us first remind that $\zeta$ must obey an incompressibility constraint given by
\begin{eqnarray*}
\int_{\p \mathcal{S}} \zeta \cdot n \d \Gamma & = & 0.
\end{eqnarray*}
Thus we can formally extend $\zeta$ in the whole domain $\mathcal{O}$ while assuming that $\div \ \zeta = 0$ in $\mathcal{F}$. That is why we can make the control appear only in the equation \eqref{control1}, the one which deals with $\mathbb{P}u$. Let us remind the linear system \eqref{premss}--\eqref{derss}:
\begin{eqnarray*}
\frac{\p u}{\p t} - \lambda u - \div \ \sigma(u,p) & = &
0, \quad \textrm{in $\mathcal{F}\times (0,\infty)$}, \\
\div \  u & = & 0, \quad \textrm{in $\mathcal{F}\times (0,\infty)$},
\end{eqnarray*}
\begin{eqnarray*}
u  =  0 , &\quad & \textrm{in $\p \mathcal{O}\times (0,\infty)$}, \\
u  =  h'(t) + \omega(t) \wedge y + \zeta(y,t) , & \quad & y \in \p \mathcal{S} ,\quad t\in (0,\infty),
\end{eqnarray*}
\begin{eqnarray*}
M h''-\lambda M h' = - \int_{\p \mathcal{S}} \sigma(u,p) n \d \Gamma , & \quad & \text{in } (0,\infty), \\
I_0\omega' - \lambda I_0\omega  = -  \int_{\p \mathcal{S}} y \wedge \sigma(u,p) n \d \Gamma  , & \quad & \text{in } (0,\infty),
\end{eqnarray*}
\begin{eqnarray*}
u(y,0)  =  u_0 (y), \  y \in \mathcal{F}, \quad
h'(0)=h_1 \in \R^3 ,\quad \omega(0) = \omega_0 \in \R^3.
\end{eqnarray*}
The boundary condition on $\p \mathcal{S}$ can be tackled by using a lifting method, like in \cite{Court} for instance. It consists in splitting the velocity $u = v + \w$ and the pressure $p = q + \pi$, so that we have
\begin{eqnarray*}
 -  \nu \Delta \w + \nabla \pi = 0, &  & \textrm{in
$\mathcal{F}$},  \\
\div \  \w = 0 , &  & \textrm{in $\mathcal{F}$},  \\
\w  =  \zeta , &  & \textrm{on $\p \mathcal{S}$}, \\
\w = 0 , &  & \textrm{on $\p \mathcal{O}$},
\end{eqnarray*}
and
\begin{eqnarray*}
\frac{\p v}{\p t} - \lambda v - \div \ \sigma(v,q)  =  -\frac{\p \w}{\p t}+\lambda \w, & \quad & \textrm{in $\mathcal{F}\times (0,\infty)$},\\
\div \  v  =  0, & \quad & \textrm{in $\mathcal{F}\times (0,\infty)$},
\end{eqnarray*}
\begin{eqnarray*}
v  =  0 , &\quad & \textrm{in $\p \mathcal{O}\times (0,\infty)$}, \\
v  =  h'(t) + \omega(t) \wedge y , & \quad & y \in \p \mathcal{S} ,\quad t\in (0,\infty),
\end{eqnarray*}
\begin{eqnarray*}
M h''(t)-\lambda M h'(t) = - \int_{\p \mathcal{S}} \sigma(v,q) n \d \Gamma - \int_{\p \mathcal{S}} \sigma(\w,\pi) n \d \Gamma,
&  & \text{in } (0,\infty), \\
I_0\omega'(t) - \lambda I_0\omega (t) = -  \int_{\p \mathcal{S}} y \wedge \sigma(v,q) n \d \Gamma -
\int_{\p \mathcal{S}} y \wedge \sigma(\w,\pi) n \d \Gamma, &  & \text{in }  t\in(0,\infty),
\end{eqnarray*}
\begin{eqnarray*}
v(y,0) & = & u_0 (y)-\w(y,0), \  y \in \mathcal{F}, \quad h'(0)=h_1 \in \R^3 ,\quad \omega(0) = \omega_0 \in \R^3.
\end{eqnarray*}
This system can be formulated as follows
\begin{eqnarray*}
\mathbb{M} V' & = & \mathbb{A} V +\lambda \mathbb{M}_0 V
+ \mathbb{B}^{(1)}\dot{\zeta} + \mathbb{B}_{\lambda}^{(0)}\zeta, \\
\left( \Id -\mathbb{P} \right) u & = & \left( \Id -\mathbb{P} \right)\left(L_0(h') + \hat{L}_0(\omega)\right),
\end{eqnarray*}
with $V = (\mathbb{P}v, h',\omega)$, $\dot{\zeta} = \frac{\p \zeta}{\p t}$ and
\begin{eqnarray*}
\mathbb{B}^{(1)}  =
\left[ \begin{array} {c}
-L_0 \\
0 \\
0
\end{array} \right],
& \quad &
\mathbb{B}_{\lambda}^{(0)}  =
\left[ \begin{array} {c}
\lambda L_0 \\
\displaystyle \mathcal{C}(-2\nu D +NA_0)L_0 \\
\displaystyle \hat{\mathcal{C}}(-2\nu D +NA_0)L_0
\end{array} \right]
.
\end{eqnarray*}
Notice that $\mathbb{M}^{-1}\mathbb{B}^{(1)} =\mathbb{B}^{(1)}$. The Duhamel's formula gives
\begin{eqnarray*}
V(t) & = & e^{t\mathcal{A}_{\lambda}}\left(U(0)+\mathbb{B}^{(1)}\zeta(\cdot,0)\right)
+ \int_0^t e^{(t-s)\mathcal{A}_{\lambda}}\left(\mathbb{M}^{-1}\mathbb{B}_{\lambda}^{(0)}\zeta+ \mathbb{B}^{(1)}\dot{\zeta} \right)\d s,
\end{eqnarray*}
and an integration by parts leads to
\begin{eqnarray*}
\int_0^t e^{(t-s)\mathcal{A}_{\lambda}}\mathbb{B}^{(1)}\dot{\zeta} \d s & = &
\mathbb{B}^{(1)} \zeta(\cdot,t) - e^{t \mathcal{A}_{\lambda}} \mathbb{B}^{(1)} \zeta(\cdot,0)
+ \int_0^t e^{(t-s)\mathcal{A}_{\lambda}}\mathcal{A}_{\lambda}\mathbb{B}^{(1)}\zeta \d s.
\end{eqnarray*}
Notice that $\mathbb{M}_0\mathbb{B}^{(1)} = 0$ and that
\begin{eqnarray*}
\mathbb{A} \mathbb{B}^{(1)} & = &
\left[ \begin{array} {c}
-A_0 L_0 \\
-\displaystyle \mathcal{C}(-2\nu D +NA_0)L_0 \\
-\displaystyle \hat{\mathcal{C}}(-2\nu D +NA_0)L_0
\end{array} \right].
\end{eqnarray*}
Thus we have
\begin{eqnarray*}
U(t) & = & e^{t\mathcal{A}_{\lambda}}U(0) +
\int_0^t e^{(t-s)\mathcal{A}_{\lambda}}\mathbb{M}^{-1}\left(\mathbb{B}_{\lambda}^{(0)} + \mathbb{A}\mathbb{B}^{(1)}\right)\zeta \d s, \\
U'(t) & = & \mathcal{A}_{\lambda} U(t) + \mathcal{B}_{\lambda} \zeta(\cdot,t) \quad \text{with } \mathcal{B}_{\lambda} = \mathbb{M}^{-1}\left(\mathbb{B}_{\lambda}^{(0)} + \mathbb{A}\mathbb{B}^{(1)}\right),
\end{eqnarray*}
and finally system \eqref{premss}--\eqref{derss} can be expressed formally in the form given by \eqref{control1}--\eqref{control2}.
\end{proof}

\subsection{Regularity of solutions of the linearized system}
\begin{proposition}
For $T>0$, $(u_0,h_1,\omega_0) \in \mathbf{H}^1_{cc}$ and $\zeta \in \L^2(0,\infty;\mathbf{H}^{3/2}(\p \mathcal{S}))$ \textcolor{black}{satisfying $\displaystyle \int_{\p \mathcal{S}} \zeta \cdot n \d \Gamma  =  0,
$} system \eqref{lin1}--\eqref{lin10} admits a unique solution $(U,H',\Omega)$ such that
\begin{eqnarray*}
U  \in  \H^{2,1}(Q_0^{T}), \quad H'  \in  \H^1(0,T;\R^3), \quad \Omega  \in  \H^1(0,T;\R^3).
\end{eqnarray*}
\end{proposition}

\begin{proof}
\textcolor{black}{Note that the formulation provided by Proposition \ref{propoperator} makes sense when we have only $\zeta \in \L^2(0,\infty;\mathbf{H}^{3/2}(\p \mathcal{S}))$.} With regards to the formulation \eqref{control1}--\eqref{control2} and the properties of the operator $\mathcal{A}$, this result can be deduced from Proposition 3.3 of \cite{Tucsnak}.
\end{proof}

\section{Approximate controllability of the homogeneous linear system} \label{secapproxcont}
In order to prove that system \eqref{lin1}--\eqref{lin10} is exponentially stabilizable, let us first show that it is approximatively controllable.
\begin{theorem} \label{thapproxcont}
System \eqref{lin1}--\eqref{lin10} is approximately controllable, in the space $\mathbf{H}^0_{cc}$ by boundary velocities $\zeta \in \L^2(0,\infty;\mathbf{H}^{3/2}(\p \mathcal{S}))$ satisfying
\begin{eqnarray*}
\int_{\p \mathcal{S}} \zeta \cdot n \d \Gamma & = & 0.
\end{eqnarray*}
\end{theorem}

\begin{proof}
Note that the operator formulation given by Proposition \ref{propoperator} does not enable us to write system \eqref{premss}--\eqref{derss} as an evolution equation. Instead of exploiting this operator formulation, we directly use the writing \eqref{premss}--\eqref{derss} \textcolor{black}{with $\lambda = 0$} and the definition of approximate controllability, as it is done in \cite{JPR3}.\\
Let us show that if $(u_0,h_1, \omega_0)=(0,0,0)$ then the reachable set $R(T)$ at time $T$, when the control $\zeta$ describes $\L^2(0,\infty;\mathbf{H}^{3/2}(\p \mathcal{S}))$ with the compatibility condition
\begin{eqnarray*}
\int_{\p \mathcal{S}} \zeta \cdot n \d \Gamma & = & 0,
\end{eqnarray*}
is dense in the space $\mathbf{L}^2(\mathcal{F})\times \R^3 \times \R^3$ that we endow with the scalar product
\begin{eqnarray*}
\langle (u,h',\omega);(\phi,k',r)\rangle & = & \int_{\mathcal{F}}u\cdot \phi + Mh'\cdot k' + I_0\omega\cdot r.
\end{eqnarray*}
For that, let $(\phi^T,k'^T,r^T)$ be in $R(T)^{\bot}$. We want to show that $(\phi^T,k'^T,r^T)=(0,0,0)$.\\
Let us introduce the adjoint system
\begin{eqnarray}
-\frac{\p \phi}{\p t}  - \div \ \sigma(\phi,\psi) = 0, & \quad & \textrm{in $\mathcal{F}\times (0,T)$}, \label{premsadj}\\
\div \  \phi = 0, & \quad & \textrm{in $\mathcal{F}\times (0,T)$},
\end{eqnarray}
\begin{eqnarray}
\phi  =  0 , &\quad & \textrm{in $\p \mathcal{O}\times (0,T)$}, \label{troisadj}\\
\phi  =  k'(t) + r(t) \wedge y  , & \quad & y \in \ \p \mathcal{S} ,\quad t\in (0,T), \label{quatreadj}
\end{eqnarray}
\begin{eqnarray}
-M k''(t) = - \int_{\p \mathcal{S}} \sigma(\phi,\psi) n \d \Gamma , & \quad & t\in (0,T),  \label{cinqsadj}\\
-I_0r'(t) = -  \int_{\p \mathcal{S}} y \wedge \sigma(\phi,\psi) n \d \Gamma , & \quad & t\in(0,T),  \label{sixsadj}
\end{eqnarray}
\begin{eqnarray}
\phi(y,T) & = & \phi^T (y), \  y \in \mathcal{F}, \quad 
k'(T)=k'^T \in \R^3 ,\quad r(T) = r^T \in \R^3. \label{dersadj}
\end{eqnarray}
By integrations by parts from systems \eqref{premss}--\eqref{derss} and \eqref{premsadj}--\eqref{dersadj}, we obtain
\begin{eqnarray*}
 \int_{\mathcal{F}}u(T)\cdot \phi^T + Mh'(T)\cdot k'^T+I_0\omega(T)\cdot r^T & = &
 -\int_0^T \int_{\p \mathcal{S}}\zeta \cdot \sigma(\phi,\psi)n \d \Gamma \d t.
\end{eqnarray*}
Thus we deduce that if $(\phi^T,k'^T,r^T)\in R(T)^{\bot}$, then we have
\begin{eqnarray}
\int_0^T  \int_{\p \mathcal{S}}\zeta \cdot \sigma(\phi,\psi)n \d \Gamma \d t & = & 0 \label{contunique}
\end{eqnarray}
for all $\zeta \in \L^2(0,\infty;\mathbf{H}^{3/2}(\p \mathcal{S}))$ such that $\displaystyle \int_{\p \mathcal{S}} \zeta \cdot n \d \Gamma  =  0$. This is equivalent to say that there exists a constant $C(t)$ such that
\begin{eqnarray}
\sigma(\phi, \psi)n  =  C(t)n & \quad & \text{on } \p \mathcal{S}. \label{condunique0}
\end{eqnarray}
We can consider $\overline{\psi} = \psi - C$ instead of $\psi$, which does not modify the system \eqref{premsadj}--\eqref{dersadj}.
Thus we get $\sigma(\phi, \psi)n =0$ in $\L^2(0,T;\mathbf{L}^2(\p \mathcal{S}))$. Then equations \eqref{cinqsadj} and \eqref{sixsadj} become
\begin{eqnarray}
k'' = 0, & \quad & r' = 0. \label{null1}
\end{eqnarray}
Now, let us look at $\phi_t = \displaystyle \frac{\p \phi}{\p t}$ and $\psi_t = \displaystyle \frac{\p \psi}{\p t}$ with the condition
\begin{eqnarray}
\sigma(\phi_t, \psi_t)n  =  0 & \quad & \text{on } \p \mathcal{S} , \label{condunique}
\end{eqnarray}
and the following homogeneous Dirichlet condition which is deduced from \eqref{troisadj} and \eqref{null1}
\begin{eqnarray*}
\phi_t  =  0 & \quad & \textrm{on $\p \mathcal{F} = \p \mathcal{O} \cup \p \mathcal{S}$}.
\end{eqnarray*}
We can use an expansion of the solution $\phi_t$ of system
\begin{eqnarray*}
-\frac{\p \phi_t}{\p t}  - \div \ \sigma(\phi_t,\psi_t) = 0, & &  \textrm{in $\mathcal{F}\times (0,T)$}, \\
\div \  \phi_t  =  0, & &  \textrm{in $\mathcal{F}\times (0,T)$}, \\
\phi_t  =  0 , & &   \textrm{in $\p \mathcal{F}\times (0,T)$},\\
\sigma(\phi_t, \psi_t)n  =  0, & &   \textrm{in $\p \mathcal{S}\times (0,T)$},
\end{eqnarray*}
in terms of the eigenfunctions of the Stokes operator, similarly as it is done in \cite{Osses1} \textcolor{black}{(see Theorem 3.1, the first part of the proof, whose arguments are valid for general geometries of $\mathcal{F}$). More precisely, considering a sequence $(\mu_j)_{j\geq 1}$ of eigenvalues of the Stokes operator, if we decompose $\phi_t(T)$ as follows
\begin{eqnarray*}
\phi_t(T) & = & \sum_{j \geq 1} a_j v_j,
\end{eqnarray*}
where the functions $v_j$ are the eigenfunctions satisfying
\begin{eqnarray*}
 - \div \ \sigma(v_j,p_j) = \mu_j v_j, & &  \textrm{in $\mathcal{F}\times (0,T)$}, \\
\div \  v_j  =  0, & &  \textrm{in $\mathcal{F}\times (0,T)$}, \\
v_j  =  0 , & &   \textrm{in $\p \mathcal{F}\times (0,T)$},
\end{eqnarray*}
\vspace*{-20pt}
\begin{eqnarray*}
\langle v_j, v_k \rangle_{\mathbf{L}^2(\mathcal{F})} = \delta_{jk}, &  & 
\langle v_j, v_k \rangle_{\mathbf{H}^1(\mathcal{F})} = 0 \text{ for } j \neq k,
\end{eqnarray*}
then we have
\begin{eqnarray*}
\phi_t = \sum_{j \geq 1} a_i e^{-\mu_j(T-t)} v_j, & & \psi_t = \sum_{j \geq 1} a_i e^{-\mu_j(T-t)} p_j.
\end{eqnarray*}
}
Thus the approximate controllability problem is reduced to showing that if
\begin{eqnarray*}
-\nu \Delta v + \nabla p = \mu v &  & \text{in } \mathcal{F}, \\
\div \ v = 0 &  & \text{in } \mathcal{F}, \\
v = 0 &  & \text{on } \p \mathcal{F}, \\
\sigma(v, p)n  = 0 &  & \text{on } \p \mathcal{S},
\end{eqnarray*}
with $\mu \in \R$, then $v=0$ in $\mathcal{F}$. Then we get the following unique continuation result (see \cite{Fabre})
\begin{eqnarray*}
\phi_t = 0 & \quad & \text{in } \mathcal{F}.
\end{eqnarray*}
Then we have only
\begin{eqnarray*}
- \div \ \sigma(\phi,\psi)  =  0, &  & \textrm{in $\mathcal{F} \times (0,T) $},\\
\div \  \phi  =  0, &  & \textrm{in $ \mathcal{F} \times (0,T)$}, \\
\phi  =  0 , &  & \textrm{in $ \p \mathcal{O} \times (0,T) $},\\
\phi  =  k' + r \wedge y  , & & (y,t) \in  \p \mathcal{S} \times (0,T).
\end{eqnarray*}
An energy estimate leads us to
\begin{eqnarray*}
2\nu \int_{\mathcal{F}} \left|D(\phi)\right|^2 \d y & = & \int_{\p \mathcal{S}}\phi \cdot \sigma(\phi, \psi)n \d \Gamma \\
& = & k'\cdot \int_{\p \mathcal{S}}\sigma(\phi,\psi)n\d \Gamma + r \cdot \int_{\p \mathcal{S}}y \wedge \sigma(\phi,\psi)n\d \Gamma.
\end{eqnarray*}
Combined to \eqref{condunique0}, we get
\begin{eqnarray*}
\int_{\mathcal{F}} \left|D(\phi)\right|^2 \d y & = & 0,
\end{eqnarray*}
and thus $D(\phi) = 0$. Using a result from \cite[page 18]{Temam}, this implies
\begin{eqnarray*}
\phi = k' + r\wedge y &  & \text{in } \mathcal{F}.
\end{eqnarray*}
The condition \eqref{troisadj} enables us to conclude
\begin{eqnarray*}
k' = 0, \ r = 0, \text{ and } \phi = 0 \text{ in } \mathcal{F}.
\end{eqnarray*}
Then the proof is completed.
\end{proof}

\begin{remark}
The adjoint system introduced in \eqref{premsadj}-\eqref{dersadj} can be written in terms of operators; Indeed, by denoting $\bphi = (\mathbb{P}\phi,k',r)^T$, we can formulate this system as follows
\begin{eqnarray*}
-\bphi ' & = & \mathcal{A}^{\ast} \bphi, \\
\bphi(T) & = & (\mathbb{P}\phi^T,k'^T,r^T)^T, \\
(\Id - \mathbb{P}) \phi & = & (\Id- \mathbb{P})\left(L_0(k') + \hat{L}_0(r)\right),
\end{eqnarray*}
where $\mathcal{A}$ is self-adjoint, so that we can write
\begin{eqnarray*}
\bphi(t) & = & e^{(T-t)\mathcal{A}_0}(\mathbb{P}\phi^T,k'^T,r^T)^T.
\end{eqnarray*}
The adjoint operator of $\mathcal{B}_0 = \mathcal{B}_{\lambda = 0}$ (whose expression is given in Proposition \ref{propoperator}) can be expressed as $\mathcal{B}^{\ast}_0 \bphi = -\sigma(z,\pi)n$, where $(z,\pi)$ is defined as the solution of
\begin{eqnarray*}
-\nu \Delta z + \nabla \pi = (-A_0)\phi & & \text{in } \mathcal{F}, \\
\div \ z = 0 & & \text{in } \mathcal{F}, \\
z = 0 & & \text{on } \p \mathcal{F} = \p \mathcal{O} \cup \p \mathcal{S}
\end{eqnarray*}
(see Lemma A.4 of \cite{JPR1} for more details). Then, proving the result stated above by using the classical characterization of approximate controllability instead of using directly the definition would lead to other difficulties.
\end{remark}

\section{Stabilization and feedback operator} \label{secfeedback}
\textcolor{black}{
\begin{theorem} \label{thstablinX}
For all $\lambda >0$ and $(u_0,h_1,\omega_0) \in \mathbf{H}^1_{cc}$, there exists a finite-dimensional subspace $\Xi$ of $\mathbf{H}^{3/2}(\p \mathcal{S})$ and a continuous linear feedback operator 
\begin{eqnarray*}
\mathcal{K}_{\lambda} : \L^2(\mathcal{F})\times \R^3 \times \R^3 \longrightarrow \Xi
\end{eqnarray*}
such that the solution $(u,h',\omega)$ of system \eqref{premss}--\eqref{derss} with $\zeta = \mathcal{K}_{\lambda}(\mathbb{P}u,h',\omega)$ satisfies
\begin{eqnarray*}
\|(u,h',\omega) \|_{\L^2(0,\infty;\mathbf{L}^2(\mathcal{F})\times \R^3 \times \R^3)} & \leq & C,
\end{eqnarray*}
for some positive constant $C$ depending only on $(u_0,h_1,\omega_0)$.
\end{theorem}
}

\begin{proof}
\textcolor{black}{Let us consider the formulation \eqref{control1}-\eqref{control2} of system \eqref{premss}--\eqref{derss} given by Proposition \ref{propoperator}, and let us focus on equation \eqref{control1}. Without loss of generality, we can choose $\lambda$ in the resolvent of $\mathcal{A}$. From Theorem \ref{thanalsg}, we know that the spectrum of $\mathcal{A}$ is reduced to a discrete set of distinct eigenvalues, that we can put in order as follows
\begin{eqnarray*}
\Re \lambda_1 \geq \Re \lambda_2 \geq \dots \geq \Re \lambda_N > - \lambda > \Re \lambda_{N+1} \geq \dots,
\end{eqnarray*}
and which are contained in an angular domain $\left\{z \in \mathbb{C}\setminus\{0\} \mid \mathrm{arg}(\alpha-z) \in (-\alpha, \alpha)\right\}$, for some $\alpha \in \R$ and $\alpha \in (0, \pi /2)$. Besides, the generalized eigenspace of each eigenvalue is of finite dimension (see \cite{Kato}, Chapter III, Theorem 6.29, page 187). Let us denote by $\Lambda(\lambda_i)$ the real generalized eigenspace associated with $\lambda_i$ if $\lambda_i \in \R$ and with the pair $(\lambda_i,\overline{\lambda_i})$ if $\Im \lambda_i \neq 0$, and let us set
\begin{eqnarray*}
\mathbf{H}_u = \bigoplus_{i=1}^N \Lambda(\lambda_i), & \quad & \mathbf{H}_s = \bigoplus_{i=N+1}^{\infty} \Lambda(\lambda_i).
\end{eqnarray*}
If $(e_j(\lambda_i))_{1 \leq j \leq m(\lambda_i)}$ is a basis of the complex generalized eigenspace associated with $\lambda_i$ (where $m(\lambda_i)$ denoting the geometric multiplicity of $\lambda_i)$, then $\Lambda(\lambda_i)$ is actually the space generated by the family $\left\{ \Re (e_j(\lambda_i)), \Im (e_j(\lambda_i)) \mid 1 \leq j \leq m(\lambda_j) \right\}$. Note that $\mathbf{H}_u$ is the unstable space of system \eqref{lin1}--\eqref{lin10}, and $\mathbf{H}_s$ is the stable one. They are both invariant by $\mathcal{A}$. Let us denote by $P_{u}$ the projection onto the finite-dimensional unstable subspace $\mathbf{H}_u$ (parallel to the stable subspace $\mathbf{H}_s$). If we project system \eqref{control1} (with $\lambda = 0$) on $\mathbf{H}_u$, we obtain
\begin{eqnarray}
\frac{\d }{\d t} P_{u}
\left( \begin{array} {lll}
\mathbb{P}u \\ h' \\ \omega
\end{array} \right)  =  \mathcal{A} P_{u}
\left( \begin{array} {lll}
\mathbb{P}u \\ h' \\ \omega
\end{array} \right) + P_{u} \mathcal{B}_{0}\zeta_0, & \quad & P_{u}
\left( \begin{array} {lll}
\mathbb{P}u(\cdot,0) \\ h'(0) \\ \omega(0)
\end{array} \right) = P_{u}
\left( \begin{array} {lll}
u_0 \\ h_1 \\ \omega_0
\end{array} \right). \nonumber \\
\label{proj1}
\end{eqnarray}
From Theorem \ref{thapproxcont}, system \eqref{lin1}--\eqref{lin10} is approximately controllable in any time $T > 0$, which means that the reachable set $R(T)$ at time $T$ is dense in $\mathbf{L}^2(\mathcal{F})\times \R^3 \times \R^3$. Hence the projected system \eqref{proj1} is also approximately controllable. The projection of the reachable set $R(T)$, through $P_{u}$, on $\mathbf{H}_u$, which is finite-dimensional, is then actually the whole space $\mathbf{H}_u$. This means that system \eqref{proj1} is exactly controllable.}\\
\textcolor{black}{Then, from \cite{Zabczyk} (see Theorem 2.9 of Chapter I, page 44), we know that there exists a linear operator $K_u$ defined on $\mathbf{H}_u$ such that the operator $\mathcal{A}P_{u}+ P_{u}\mathcal{B}_{0}K_u$ is exponentially stable with $\lambda$ as decay rate. Since the operator $\mathcal{A}(\Id-P_{u})$ has the same property, it is sufficient to set
\begin{eqnarray*}
\mathcal{K}_{\lambda} =  K_u P_{u} ,
\text{ and to choose $\Xi$ such that: } \mathrm{Im}\left(P_{u}\mathcal{B}_0 \right) = P_{u}\mathcal{B}_0(\Xi).
\end{eqnarray*}
Thus there exists a constant $C>0$ depending only on $(u_0, h_1, \omega_0)$ such that the solution of system \eqref{control1} with $\lambda = 0$ and $\zeta_0=\mathcal{K}_{\lambda}(\mathbb{P}u,h',\omega)^T$ satisfies
\begin{eqnarray*}
\|(\mathbb{P}u(\cdot,t),h'(t),\omega(t)) \|_{\mathbf{L}^2(\mathcal{F})\times \R^3 \times \R^3} & \leq & C\exp(-\lambda t).
\end{eqnarray*}
On the other hand, with regards to the equation \eqref{control2}, the same kind of inequality holds for $((\Id-\mathbb{P})u,h',\omega)$. Then we notice that $(u,h',\omega)^T$ is the solution of system \eqref{control1} corresponding to $\lambda = 0$ and $\zeta_0=\mathcal{K}_{\lambda}(\mathbb{P}u,h',\omega)^T$ if and only if $(\hat{u}, \hat{h}', \hat{\omega}) = e^{\lambda t} (u,h',\omega)$ is the solution of system \eqref{control1} with $\zeta = e^{\lambda t}\zeta_0$. Hence it is sufficient to consider the operator $e^{\lambda t} \mathcal{K}_{\lambda}$ instead of $\mathcal{K}_{\lambda}$ to complete the proof.}
\end{proof}

\begin{remark}
\textcolor{black}{Let us add some comments on a practical means of constructing a feedback operator as given in the proof above. Let us consider $\mathcal{P}^{\infty}_{(u_0,h_1,\omega_0)}$ the following infinite time horizon control problem
\begin{eqnarray*}
\inf \left\{\begin{small} \mathcal{J}(\mathbb{P}u, h', \omega, \zeta) \mid (\mathbb{P}u, h', \omega, \zeta) \text{ obeys \eqref{control1} with } \zeta \in \L^2(0,\infty;\Xi), \ \int_{\p \mathcal{S}} \zeta\cdot n \d \Gamma = 0\end{small}
\right\}
\end{eqnarray*}
where
\begin{eqnarray*}
 \mathcal{J}(\mathbb{P}u, h', \omega, \zeta) = \frac{1}{2} \int_0^{\infty} \| P_{u}(\mathbb{P}u(\cdot,t), h'(t), \omega(t), \zeta(\cdot,t)) \|^2_{\mathbf{H}_u}\d t + \frac{1}{2} \int_0^{\infty} \| \zeta(\cdot,t) \|^2_{\Xi}\d t.
\end{eqnarray*}
This approach is quite classical for this kind of feedback stabilization problem. Like it is done in \cite{Sontag} (see Lemma 8.4.1 page 381 and Theorem 41 page 384) for instance, the problem $\mathcal{P}^{\infty}_{(u_0,h_1,\omega_0)}$ admits a unique solution $(\mathbb{P}u,h',\omega)$, since we have seen in the proof above that the projected system is controllable (for $\lambda =0$, but also for all $\lambda >0$). There exists an operator $\Pi_u \in \mathcal{L}(\mathbf{H}_u, (\mathbf{H}_u)^{\ast})$ satisfying $\Pi_u = \Pi_u^{\ast} \geq 0$ (in the sense of quadratic forms) such that the optimal cost is written as
\begin{eqnarray*}
\frac{1}{2}\left( P_{u}(\mathbb{P}u_0, h_1, \omega_0), \Pi_u P_{u}(\mathbb{P}u_0, h_1, \omega_0) \right)_{\mathbf{H}_u, (\mathbf{H}_u)^{\ast}}.
\end{eqnarray*}
Moreover, $\Pi_u$ is the solution of the following finite dimensional algebraic Riccati equation
\begin{eqnarray*}
\Pi_u \in \mathcal{L}(\mathbf{H}_u, (\mathbf{H}_u)^{\ast}), \ \Pi_u = \Pi_u^{\ast} \geq 0, \quad
\Pi_u \mathcal{A}_{\lambda,u} + \mathcal{A}_{\lambda,u}^{\ast} \Pi_u^{\ast} - \Pi_u \mathcal{B}_{\lambda,u}^{\ast} \mathcal{A}_{\lambda,u}^{\ast} \Pi_u^{\ast} + I_u = 0,
\end{eqnarray*}
where
\begin{eqnarray*}
\mathcal{A}_{\lambda,u} & = & P_{u}\mathcal{A}_{\lambda}P_{u} \in \mathcal{L}(\mathbf{H}_u, \mathbf{H}_u), \\
\mathcal{B}_{\lambda,u} & = & P_{u}\mathcal{B}_{\lambda} \in \mathcal{L}(\Xi, \mathbf{H}_u),
\end{eqnarray*}
and the expression of the optimal closed-loop control is
\begin{eqnarray*}
\zeta  =  \mathcal{K}_{\lambda} (\mathbb{P}u,h',\omega) = -\mathcal{B}^{\ast}_{\lambda,u}\Pi_u P_{u}(\mathbb{P}u,h',\omega).
\end{eqnarray*}
}
\end{remark}

\begin{remark}
\textcolor{black}{Let us now take a look at the regularity of the boundary feedback control obtained above. The maximal regularity in space considered for the control $\zeta$ until now has been $\mathbf{H}^{3/2}(\p \mathcal{S})$. In the perspective of the nonlinear system studied in Part II, we will actually need to consider more regularity, namely
\begin{eqnarray*}
\zeta & \in & \L^2(0,\infty;\mathbf{H}^{5/2}(\p \mathcal{S}))\cap \H^1(0,\infty;\mathbf{H}^{1/2}(\p \mathcal{S})).
\end{eqnarray*}
Note first that the matrix $\Pi_u$ given in the remark above as the solution of an algebraic Riccati equation does not depend on time. This is then also the case of $\mathcal{K}_{\lambda} \in \mathcal{L}(\mathbf{H}_u, \Xi)$. Thus, for
\begin{eqnarray*}
(u,h',\omega) & \in & \H^{2,1}(Q_{\infty}^0)\times \H^1(0,\infty;\R^3) \times \H^1(0,\infty;\R^3),
\end{eqnarray*}
we have in particular
\begin{eqnarray*}
\mathcal{K}_{\lambda}(\mathbb{P}u,h',\omega) & \in & \H^1(0,\infty; \Xi),\\
\| \mathcal{K}_{\lambda}(\mathbb{P}u,h',\omega) \|_{\H^1(0,\infty; \Xi)} & \leq & 
\| (\mathbb{P}u,h',\omega) \|_{\H^1(0,\infty;\mathbf{H}_u)}.
\end{eqnarray*}
Let us remind that the eigenspace $\mathbf{H}_u$ is a finite dimensional subspace of $\L^2(\mathcal{F})\times \R^3 \times \R^3$, and that $\Xi$ is a finite dimensional subspace of $\mathbf{H}^{3/2}(\p \mathcal{S})$. The basis functions of $\mathbf{H}_u$ can be thus chosen regular enough to enable us to choose $\Xi$ as being a subspace of $\mathbf{H}^{5/2}(\p \mathcal{S})$, and so well that we can deduce the estimate
\begin{eqnarray*}
\| \mathcal{K}_{\lambda}(\mathbb{P}u,h',\omega) \|_{\H^1(0,\infty; \mathbf{H}^{5/2}(\p \mathcal{S}))} & \leq & 
\| (\mathbb{P}u,h',\omega) \|_{\H^{2,1}(Q_{\infty}^0)\times \H^1(0,\infty;\R^3) \times \H^1(0,\infty;\R^3)}.
\end{eqnarray*}
In fact, in Part II, we will only need the regularity
\begin{eqnarray*}
\zeta = \mathcal{K}_{\lambda}(\mathbb{P}u,h',\omega) & \in & \L^2(0,\infty;\mathbf{H}^{5/2}(\p \mathcal{S})) \cap \H^1(0,\infty;\mathbf{H}^{1/2}(\p \mathcal{S})).
\end{eqnarray*}
}
\end{remark}

\section{Definition of an admissible deformation from a boundary velocity} \label{secAhAh}
Given a control $\zeta \in \L^2(0,\infty;\mathbf{H}^{5/2}(\p \mathcal{S})) \cap \H^1(0,\infty;\mathbf{H}^{1/2}(\p \mathcal{S}))$ with the compatibility condition
\begin{eqnarray*}
\int_{\p \mathcal{S}} \zeta \cdot n \d \Gamma & = & 0,
\end{eqnarray*}
(which can be chosen in a feedback form, as explained in the previous section), we want to define an internal deformation of the solid, denoted by $X^{\ast}_{\zeta}$, which is {\it admissible} in the sense of Definition \ref{deflincontrol}. In particular we look for $X^{\ast}_{\zeta}$ which satisfies the linearized constraints given by \eqref{const105}--\eqref{const102}, and such that
\begin{eqnarray*}
 e^{\lambda t} \frac{\p X^{\ast}_{\zeta}}{\p t}(y,t) & = & \zeta(y,t), \quad (y,t) \in \p \mathcal{S} \times (0,\infty).
\end{eqnarray*}
We also want the norm of $e^{\lambda t} \frac{\p X^{\ast}_{\zeta}}{\p t}$ in the space
\begin{eqnarray*}
\L^2(0,\infty;\mathbf{H}^3(\mathcal{S}) \cap \H^1(0,\infty;\mathbf{H}^1(\mathcal{S}))
\end{eqnarray*}
to be controlled by the norm of $\zeta$ in $\L^2(0,\infty;\mathbf{H}^{5/2}(\p \mathcal{S})) \cap \H^1(0,\infty;\mathbf{H}^{1/2}(\p \mathcal{S}))$.

\subsection{Definition of a deformation satisfying the linearized constraints} \label{secchoice1} \label{secLame}
We search for the deformation $X^{\ast}_{\zeta}$ by writing it as
\begin{eqnarray}
X^{\ast}_{\zeta}(y,t) & = & y + \int_0^t e^{-\lambda s}\varphi(y,s) \d s, \label{eqintegrale}
\end{eqnarray}
where the velocity $\varphi(\cdot,t)$ is the solution of the following elliptic system
\begin{eqnarray}
\mu \varphi - 2\div \ D(\varphi) = F(\varphi) & & \text{in } \mathcal{S}, \label{Lame01} \\
\varphi =  \zeta & & \text{on } \p \mathcal{S}, \label{Lame04}
\end{eqnarray}
with
\begin{eqnarray*}
& & \mu \in \R, \quad D(\varphi) = \frac{1}{2}\left(\nabla \varphi + \nabla \varphi ^T \right), \\
& & F(\varphi)(y,t) = \frac{\rho_{\mathcal{S}}}{M}\left(\int_{\p \mathcal{S}} 2D(\varphi) n \d \Gamma \right) +
\rho_{\mathcal{S}}I_0^{-1}\left(\int_{\p \mathcal{S}} y \wedge 2D(\varphi) n \d \Gamma \right) \wedge y.
\end{eqnarray*}

\begin{proposition} \label{lemmaxistence}
For $\zeta \in \L^2(0,\infty;\mathbf{H}^{5/2}(\mathcal{S})) \cap \H^1(0,\infty;\mathbf{H}^{1/2}(\mathcal{S}))$ satisfying
\begin{eqnarray*}
\int_{\p \mathcal{S}} \zeta \cdot n \d \Gamma & = & 0,
\end{eqnarray*}
system \eqref{Lame01}--\eqref{Lame04} admits a unique solution $\varphi$ in $\L^2(0,\infty;\mathbf{H}^3(\mathcal{S}) \cap \H^1(0,\infty;\mathbf{H}^1(\mathcal{S}))$, for $\mu>0$ large enough. Moreover, there exists a positive constant $C > 0$ such that
\begin{eqnarray}
\| \varphi \|_{\L^2(0,\infty;\mathbf{H}^3(\mathcal{S}) \cap \H^1(0,\infty;\mathbf{H}^1(\mathcal{S}))}
& \leq & C \| \zeta \|_{\L^2(0,\infty;\mathbf{H}^{5/2}(\mathcal{S})) \cap \H^1(0,\infty;\mathbf{H}^{1/2}(\mathcal{S}))}.
\label{estinftylame}
\end{eqnarray}
Besides, if $\zeta_1, \zeta_2 \in \L^2(0,\infty;\mathbf{H}^{5/2}(\mathcal{S})) \cap \H^1(0,\infty;\mathbf{H}^{1/2}(\mathcal{S}))$, and if $\varphi_1$ and $\varphi_2$ denote the solutions associated with $\zeta_1$ and $\zeta_2$ respectively, then
\begin{eqnarray}
\| \varphi_2 - \varphi_1 \|_{\L^2(0,\infty;\mathbf{H}^3(\mathcal{S}) \cap \H^1(0,\infty;\mathbf{H}^1(\mathcal{S}))} & \leq &
C \| \zeta_2 - \zeta_1 \|_{\L^2(0,\infty;\mathbf{H}^{5/2}(\mathcal{S})) \cap \H^1(0,\infty;\mathbf{H}^{1/2}(\mathcal{S}))}. \nonumber \\
\label{estinftylame12}
\end{eqnarray}
\end{proposition}

\begin{proof}
The proof of this proposition is given \textcolor{black}{below} in section \ref{lemma52}.
\end{proof}

\begin{remark}
The compatibility condition assumed for the datum $\zeta$ is useless for the proof of Proposition \ref{lemmaxistence}, but contributes to making the mapping $X^{\ast}_{\zeta}$ so obtained an {\it admissible} control (in the sense of Definition \ref{deflincontrol}).
\end{remark}

Let us see that the mapping $X^{\ast}_{\zeta}$ thus chosen is admissible.

\begin{corollary} \label{lemmadmissible}
For $\zeta \in \L^2(0,\infty;\mathbf{H}^{5/2}(\mathcal{S})) \cap \H^1(0,\infty;\mathbf{H}^{1/2}(\mathcal{S}))$, the deformation $X^{\ast}_{\zeta}$ provided by Proposition \ref{lemmaxistence} and equation \eqref{eqintegrale} is admissible for the linear system \eqref{lin1}--\eqref{lin10} in the sense of Definition \ref{deflincontrol}.
\end{corollary}

\begin{proof}
The constraints imposed in Definition \ref{deflincontrol} are equivalent to the following ones expressed in term of $\varphi$:
\begin{eqnarray*}
\int_{\p \mathcal{S}} \varphi \cdot n \d \Gamma = 0, \quad \int_{\mathcal{S}} \varphi \d y = 0 , \quad
\int_{\mathcal{S}} y \wedge \varphi \d y = 0.
\end{eqnarray*}
Thus we have to verify that the mapping $\varphi$ solution of \eqref{Lame01}--\eqref{Lame04} satisfies these constraints. The first constraint, which corresponds to \eqref{const105}, is satisfied thanks to the compatibility condition assumed for $\zeta$. For the two other ones, let us remind that we have
\begin{eqnarray*}
\int_{\mathcal{S}} y \d y & = & 0,
\end{eqnarray*}
and since the tensor $D(\varphi)$ is symmetric, we also have
\begin{eqnarray*}
y \wedge \div D(\varphi) & = & \div\left(\mathbb{S}(y) D(\varphi)\right).
\end{eqnarray*}
Then Equation \eqref{Lame01} leads us to
\begin{eqnarray*}
\mu \int_{\mathcal{S}} \varphi \d y & = & 2\left(\int_{\p \mathcal{S}} D(\varphi) n \d \Gamma \right)
- 2\int_{\mathcal{S}}\div\left(D( \varphi) \right) \d y, \\
\mu \int_{\mathcal{S}} y \wedge \varphi \d y & = & 2\left(\int_{\p \mathcal{S}} y \wedge D(\varphi) n \d \Gamma \right)
- 2\int_{\mathcal{S}}\div\left(\mathbb{S}(y) D( \varphi ) \right) \d y, \\
\end{eqnarray*}
and thus by using the divergence formula we get the two other constraints.
\end{proof}

\subsection{Proof of Proposition \ref{lemmaxistence}} \label{lemma52}
Instead of solving directly system \eqref{Lame01}--\eqref{Lame04}, let us first consider a lifting of the nonhomogeneous Dirichlet condition. We set $\w$ the solution of the following Dirichlet problem
\begin{eqnarray*}
\div \ \w = 0 & & \text{in } \mathcal{S}, \\
\w = \zeta & & \text{on } \p \mathcal{S},
\end{eqnarray*}
with the classical estimates (see \cite{Galdi1}, the nonhomogeneous Dirichlet condition can be lifted by Theorem 3.4 of Chapter II, and the resolution made by using Exercise 3.4 and Theorem 3.2 of Chapter III):
\begin{eqnarray*}
\| \w \|_{\mathbf{H}^2(\mathcal{S})} & \leq & C\| \zeta \|_{\mathbf{H}^{3/2}(\p \mathcal{S})}, \\
\| \w \|_{\mathbf{H}^3(\mathcal{S})} & \leq & C\| \zeta \|_{\mathbf{H}^{5/2}(\p \mathcal{S})}, \\
\| \w_t \|_{\mathbf{H}^1(\mathcal{S})} & \leq & C\| \zeta_t \|_{\mathbf{H}^{1/2}(\p \mathcal{S})}.
\end{eqnarray*}
Then by setting $\phi = \varphi - \w$, we are interested in solving the following system
\begin{eqnarray*}
\mu \phi - 2\div \ D(\phi) = F(\phi) - \mu \w + \Delta \w + F(\w) & & \text{in } \mathcal{S},  \\
\phi =  0 & & \text{on } \p \mathcal{S},
\end{eqnarray*}
for some $\mu > 0$ large enough, in the space $\mathbf{H}^2(\mathcal{S})$ in a first time. A solution of this system can be obtained as a fixed point of the following mapping
\begin{eqnarray*}
\begin{array} {rrcl}
\mathbf{N} \ : \ & \mathbf{H}^2(\mathcal{S}) & \rightarrow & \mathbf{H}^2(\mathcal{S}) \\
& \psi & \mapsto & \phi,
\end{array}
\end{eqnarray*}
where $\phi$ is the solution of the classical elliptic system
\begin{eqnarray}
\mu \phi - 2\div \ D(\phi) = F(\psi) - \mu \w + 2\div \ D(\w) + F(\w) & & \text{in } \mathcal{S}, \label{eqlamex1} \\
\phi =  0 & & \text{on } \p \mathcal{S}.
\end{eqnarray}
For proving that this mapping is well-defined, let us give some preliminary estimates. The equality
\begin{eqnarray*}
2 D(\phi):D(\phi) - \nabla \phi : \nabla \phi & = &
\div \left( (\phi\cdot \nabla)\phi - (\div \ \phi) \phi \right) + \left( \div \ \phi \right)^2
\end{eqnarray*}
leads in $\mathbf{H}^1_0(\mathcal{S})$ to
\begin{eqnarray*}
\| \nabla \phi \|^2_{\mathbf{L}^2(\mathcal{S})} & \leq & 2 \| D (\phi) \|^2_{\mathbf{L}^2(\mathcal{S})},
\end{eqnarray*}
and then the Poincar\'e inequality provides a positive constant $C_p$ such that
\begin{eqnarray*}
\| \phi \|_{\mathbf{H}^1(\mathcal{S})} & \leq & C_p \| D(\phi) \|_{\mathbf{L}^2(\mathcal{S})}.
\end{eqnarray*}
We can estimate the norm $\mathbf{H}^2(\mathcal{S})$ as follows
\begin{eqnarray}
\| \phi \|^2_{\mathbf{H}^2(\mathcal{S})} & \leq &
C_1\left( \| \div D( \phi) \|^2_{\mathbf{L}^2(\mathcal{S})} + \| \phi \|^2_{\mathbf{H}^1(\mathcal{S})} \right), \nonumber \\
& \leq & C_1\left( \| \div D( \phi) \|^2_{\mathbf{L}^2(\mathcal{S})} + C_p^2\| D( \phi) \|^2_{\mathbf{L}^2(\mathcal{S})} \right). \label{refH2}
\end{eqnarray}
The trace of $D (\psi) $ on $\p \mathcal{S}$ which appears in the expression of $F(\psi)$ can be estimated as follows
\begin{eqnarray*}
\| D( \psi) n \|_{\mathbf{L}^2(\p \mathcal{S})}
& \leq &  C_2 \| \psi \|_{\mathbf{H}^{3/2+\varepsilon}(\mathcal{S})} , \\
& \leq &  C_2 \| \psi \|^{\alpha}_{\mathbf{H}^{2}(\mathcal{S})} \| \psi \|^{1-\alpha}_{\mathbf{H}^{1}(\mathcal{S})}, \\
& \leq &  C_2 C_p \| \psi \|^{\alpha}_{\mathbf{H}^{2}(\mathcal{S})} \| D(\psi) \|^{1-\alpha}_{\mathbf{L}^{2}(\mathcal{S})},
\end{eqnarray*}
with $\alpha = 1/2 + \varepsilon$, for some $\varepsilon >0$ which can be chosen small enough. Thus, by taking the inner product of the equality \eqref{eqlamex1} by $\div D( \phi)$, we obtain
\begin{eqnarray*}
\mu \| D(\phi) \|^2_{\mathbf{L}^{2}(\mathcal{S})} + 2\| \div D( \phi) \|^2_{\mathbf{L}^{2}(\mathcal{S})}
& \leq & C\left(\| F(\psi) \|_{\mathbf{L}^{2}(\mathcal{S})}
+ \| \w \|_{\mathbf{H}^{2}(\mathcal{S})} \right) \| \div D( \phi) \|_{\mathbf{L}^{2}(\mathcal{S})} \\
& \leq & \tilde{C}\left(\| D( \psi) n \|^2_{\mathbf{L}^{2}(\p \mathcal{S})} + \| \w \|^2_{\mathbf{H}^{2}(\mathcal{S})} \right)
 + \frac{1}{2} \| \div D( \phi) \|^2_{\mathbf{L}^{2}(\mathcal{S})}, \\
\mu \| D(\phi) \|^2_{\mathbf{L}^{2}(\mathcal{S})} + \frac{3}{2}\| \div D( \phi) \|^2_{\mathbf{L}^{2}(\mathcal{S})}
& \leq & C_3\left(\| \psi \|^{2\alpha}_{\mathbf{H}^{2}(\mathcal{S})} \| D(\psi) \|^{2-2\alpha}_{\mathbf{L}^{2}(\mathcal{S})}
+ \| \w \|^2_{\mathbf{H}^{2}(\mathcal{S})} \right).
\end{eqnarray*}
By using \eqref{refH2} in the left-hand-side, it gives
\begin{eqnarray*}
(2\mu - 3C_p^2)C_1 \| D(\phi) \|^2_{\mathbf{L}^{2}(\mathcal{S})} + 3\|  \phi \|^2_{\mathbf{H}^{2}(\mathcal{S})}
& \leq & 2C_1 C_3\left(\| \psi \|^{2\alpha}_{\mathbf{H}^{2}(\mathcal{S})} \| D(\psi) \|^{2-2\alpha}_{\mathbf{L}^{2}(\mathcal{S})}
+ \| \w \|^2_{\mathbf{H}^{2}(\mathcal{S})} \right).
\end{eqnarray*}
We now use the Young inequality on the right-hand-side, by introducing some $\delta > 0$ which can be chosen as small as desired, as follows
\begin{eqnarray*}
2 C_1 C_3 \| D(\psi) \|^{2-2\alpha}_{\mathbf{L}^{2}(\mathcal{S})} \| \psi \|^{2\alpha}_{\mathbf{H}^{2}(\mathcal{S})}
& \leq & \frac{\delta^p}{p}\| \psi \|^{2}_{\mathbf{H}^{2}(\mathcal{S})} +
\left(\frac{2C_1C_3}{\delta}\right)^q \frac{1}{q}  \| D(\psi) \|^{2}_{\mathbf{L}^{2}(\mathcal{S})}
\end{eqnarray*}
with $p = 2/(1+2\varepsilon)$ and $q = 2/(1-2\varepsilon)$. Then we have
\begin{eqnarray*}
(2\mu - 3C_p^2)C_1 \| D(\phi) \|^2_{\mathbf{L}^{2}(\mathcal{S})} + 3\|  \phi \|^2_{\mathbf{H}^{2}(\mathcal{S})} & \leq &
\frac{\delta^p}{p}\| \psi \|^{2}_{\mathbf{H}^{2}(\mathcal{S})} +
\left(\frac{2C_1C_3}{\delta}\right)^q \frac{1}{q}  \| D(\psi) \|^{2}_{\mathbf{L}^{2}(\mathcal{S})} \\
 & & + C \| \zeta \|_{\mathbf{H}^{3/2}(\p \mathcal{S})}.
\end{eqnarray*}
Thus, by setting
\begin{eqnarray*}
\mathbf{B}_R & = & \left\{  \phi \in \mathbf{H}^2(\mathcal{S}) \mid
(2\mu - 3C_p^2)C_1 \| D(\phi) \|^2_{\mathbf{L}^{2}(\mathcal{S})} + \|  \phi \|^2_{\mathbf{H}^{2}(\mathcal{S})} \leq R \right\},
\end{eqnarray*}
and by choosing $\delta >0$ small enough, and $\mu >0$ and $R>0$ large enough, we can see that the ball $\mathbf{B}_R$ is stable by the mapping $\mathbf{N}$. By the same inequalities we can see that $\mathbf{N}$ is a contraction in $\mathbf{B}_R$, and thus $\mathbf{N}$ admits a unique fixed point in $\mathbf{H}^2(\mathcal{S})$.\\
The same method can be applied in order to prove the regularity in $\mathbf{H}^3(\mathcal{S})$. Indeed, since we have the equality $\nabla (\div D(\phi)) = \div D(\nabla \phi)$, the gradient satisfies a similar equality, as follows
\begin{eqnarray*}
\mu \nabla \phi - 2\div D(\nabla \phi) = \nabla F(\phi) - \mu \nabla \w + 2\div D(\nabla \w) + \nabla F(\w) &  & \text{in } \mathcal{S},
\end{eqnarray*}
so that we can show that $\nabla \phi$ lies in $\mathbf{H}^2(\mathcal{S})$. Then we have the estimate
\begin{eqnarray*}
\| \phi \|_{\mathbf{H}^3(\mathcal{S})}  & \leq & C\| \w \|_{\mathbf{H}^3(\mathcal{S})} ,
\end{eqnarray*}
and since $\varphi =  \phi + \w$, we have
\begin{eqnarray*}
\| \varphi \|_{\mathbf{H}^3(\mathcal{S})}
& \leq & \tilde{C}\| \w \|_{\mathbf{H}^3(\mathcal{S})}  \leq \overline{C} \| \zeta \|_{\mathbf{H}^{5/2}(\p \mathcal{S})}, \\
\| \varphi \|_{\L^2(0,\infty;\mathbf{H}^3(\mathcal{S}))}
& \leq & \overline{C} \| \zeta \|_{\L^2(0,\infty;\mathbf{H}^{5/2}(\p \mathcal{S}))}.
\end{eqnarray*}
The estimate which deals with the time-derivative of $\phi$ can be obtained easily. Indeed, by taking the inner scalar product of the equality
\begin{eqnarray*}
\mu \phi_t - \Delta \phi_t & = & F(\phi_t) + \w_t
\end{eqnarray*}
by $\phi_t$, we notice that the contribution of the right-hand-side force vanishes, as follows
\begin{eqnarray*}
\int_{\mathcal{S}} F(\phi_t) \cdot \phi_t & = & 0,
\end{eqnarray*}
because $\phi_t$ satisfies the constraints
\begin{eqnarray*}
\int_{\mathcal{S}} \phi_t = 0, & \quad & \int_{\mathcal{S}}y \wedge \phi_t = 0.
\end{eqnarray*}
By this means we get easily
\begin{eqnarray*}
\| \varphi \|_{\H^1(0,\infty;\mathbf{H}^1(\mathcal{S}))}
& \leq & C \| \zeta \|_{\H^1(0,\infty;\mathbf{H}^{1/2}(\p \mathcal{S}))},
\end{eqnarray*}
\textcolor{black}{and thus the proof is complete.}

\medskip
Received xxxx 20xx; revised xxxx 20xx.
\medskip

\end{document}